\documentclass[11pt]{amsart}

\usepackage[english]{babel}
\usepackage{amsmath}
\usepackage{amsthm}
\usepackage{amssymb}
\usepackage{amsfonts}
\usepackage{amsxtra}
\usepackage{color}
\usepackage[breaklinks]{hyperref}
\usepackage{enumerate}

\numberwithin{equation}{section}

\newcommand{\td}{\,\mathrm{d}}
\newcommand{\Ad}{\textup{Ad}}

\renewcommand\Re{\operatorname{Re}}
\renewcommand\Im{\operatorname{Im}}
\newcommand{\id}{\textup{id}}
\newcommand{\Sym}{\textup{Sym}}
\newcommand{\Skew}{\textup{Skew}}
\newcommand{\Herm}{\textup{Herm}}
\newcommand{\GL}{\textup{GL}}

\newcommand{\SL}{\textup{SL}}
\renewcommand{\sl}{\mathfrak{sl}}
\newcommand{\upS}{\textup{S}}
\newcommand{\Sp}{\textup{Sp}}
\renewcommand{\sp}{\mathfrak{sp}}
\newcommand{\upO}{\textup{O}}
\newcommand{\SO}{\textup{SO}}
\newcommand{\so}{\mathfrak{so}}
\newcommand{\Spin}{\textup{Spin}}

\newcommand{\upU}{\textup{U}}
\newcommand{\SU}{\textup{SU}}
\newcommand{\su}{\mathfrak{su}}
\newcommand{\upF}{\textup{F}}

\newcommand{\RR}{\mathbb{R}}

\newcommand{\CC}{\mathbb{C}}
\newcommand{\ZZ}{\mathbb{Z}}
\newcommand{\NN}{\mathbb{N}}

\newcommand{\HH}{\mathbb{H}}
\newcommand{\OO}{\mathbb{O}}

\newcommand{\FF}{\mathbb{F}}

\newcommand{\1}{\mathbf{1}}
\newcommand{\0}{\mathbf{0}}
\newcommand{\Ind}{\textup{Ind}}

\newcommand{\calE}{\mathcal{E}}
\newcommand{\calO}{\mathcal{O}}

\newcommand{\calD}{\mathcal{D}}

\newcommand{\frakg}{\mathfrak{g}}

\newcommand{\frakk}{\mathfrak{k}}
\newcommand{\frakp}{\mathfrak{p}}
\newcommand{\fraks}{\mathfrak{s}}
\newcommand{\frakn}{\mathfrak{n}}
\newcommand{\fraka}{\mathfrak{a}}

\newcommand{\frakm}{\mathfrak{m}}

\newcommand{\frakh}{\mathfrak{h}}

\newcommand{\Det}{\textup{Det}}

\newcommand{\pr}{\textup{pr}}

\newcommand{\diag}{\textup{diag}}

\DeclareMathOperator{\Gr}{Gr}
\newcommand{\Hom}{\textup{Hom}}

\newcommand{\blank}{\,\cdot\,}

\newcommand{\Rspan}{\RR_+\textup{-\,span}}

\DeclareMathOperator{\Stab}{Stab}
\newcommand{\reg}{{\textup{reg}}}

\theoremstyle{plain}
\newtheorem{theorem}{Theorem}[section]
\newtheorem{proposition}[theorem]{Proposition}
\newtheorem{lemma}[theorem]{Lemma}
\newtheorem{corollary}[theorem]{Corollary}

\newtheorem{fact}[theorem]{Fact}
\newtheorem{problem}{Problem}
\newtheorem{thmalph}{Theorem}

\theoremstyle{definition}

\newtheorem{example}[theorem]{Example}

\newtheorem{remark}[theorem]{Remark}

\begin{document}

\title[Knapp--Stein type intertwining operators for symmetric pairs]{Knapp--Stein type intertwining operators for symmetric pairs}
\date{January 7, 2015}

\author{Jan M\"{o}llers}
\author{Bent \O rsted}
\author{Yoshiki Oshima}

\address{Department of Mathematics, The Ohio State University, 231 West 18th Avenue, Columbus, OH 43210, USA}
\email{mollers.1@osu.edu}

\address{Institut for Matematiske Fag, Aarhus Universitet, Ny Munkegade 118, 8000 Aarhus C, Denmark}
\email{orsted@imf.au.dk}

\address{Kavli IPMU (WPI), The University of Tokyo, 5-1-5 Kashiwanoha, Kashiwa, 277-8583, Japan}
\email{yoshiki.oshima@ipmu.jp}
\begin{abstract}
For a symmetric pair $(G,H)$ of reductive groups we construct a family of intertwining operators between spherical principal series representations of $G$ and $H$ that are induced from parabolic subgroups satisfying certain compatibility conditions. The operators are given explicitly in terms of their integral kernels and we prove convergence of the integrals for an open set of parameters and meromorphic continuation. We further discuss uniqueness of intertwining operators, and for the rank one cases
$$ (G,H)=(\SU(1,n;\FF),\upS(\upU(1,m;\FF)\times\upU(n-m;\FF))), \quad \FF=\RR,\CC,\HH,\OO,  $$
and for the pair
$$ (G,H)=(\GL(4n,\RR),\GL(2n,\CC)) $$
we show that for a certain choice of maximal parabolic subgroups our operators generically span the space of intertwiners.
\end{abstract}



\maketitle

\section*{Introduction}

Intertwining operators of various forms have been a cornerstone of group representation theory, both for classical applications in physics, for understanding the structure of induced representations, and more recently in connection with the study of branching laws. For a unitary representation $\pi$ of a Lie group $G$, the branching with respect to a closed subgroup $H$ means considering the restriction of $\pi$ to $H$ and finding its irreducible constituents. One may ask similar questions for the category of smooth representations, and alternatively with reductive groups for the algebraic category of Harish-Chandra modules. In general such problems are very complicated, and one has to restrict to subclasses of groups and representations where useful answers are to be found. T. Kobayashi~\cite{Kob13} has as part of his program introduced the notion of {\it symmetry breaking operators}, much in the spirit of the notion of {\it symmetry breaking} in physics; these are operators in the space $\Hom_H(\pi|_H,\tau)$ for representations $\pi$ of $G$ and $\tau$ of $H$, say in the smooth category. He posed the following problem:
\begin{problem}[{\cite[Problem A]{Kob13}}]\label{prb:Intro1}
Construct explicitly \textit{symmetry breaking operators} in $\Hom_H(\pi|_H,\tau)$, and classify them.
\end{problem}
For several pairs $(G,H)$ of classical groups Sun-Zhu~\cite{SZ12} recently showed that the space $\Hom_H(\pi|_H,\tau)$ is at most one-dimensional for all irreducible smooth representations $\pi$ and $\tau$ of Casselman-Wallach type (see also \cite{AGRS10} and references therein). General bounds on the number of \textit{symmetry breaking operators} are proven by Kobayashi--Oshima~\cite{KO11}. A first example with a complete description of $\Hom_H(\pi|_H,\tau)$ for $\pi$ and $\tau$ in a certain subclass of representations is given by Kobayashi--Speh~\cite{KS,KS14}. Other examples of \textit{symmetry breaking operators} are differential operators such as the Juhl operators or the Rankin--Cohen brackets and their generalizations, see e.g.\ \cite{BC12,KOSS,KP13} and references therein.

In the present paper we shall explicitly construct a family of \textit{symmetry breaking operators} in the setting where $(G, H)$ is a reductive symmetric pair. These operators are natural extensions of the Knapp--Stein operators~\cite{KS71,KS80} intertwining between parabolically induced representations, in our case from representations of $G$ to representations of $H$. As the classical Knapp--Stein operators, our new \textit{symmetry breaking operators} are singular integral operators which we define in terms of their integral kernel. Our construction generalizes operators previously studied in the context of invariant trilinear forms as well as a family of operators studied in \cite{KS,KS14} for rank one orthogonal groups.

While we have to impose some technical conditions on the groups in question, the basic idea of our construction is rather simple, and we have tried to make it accessible in the spirit of Knapp--Stein operators. Also, we have indicated a number of examples and future directions of research; we expect that our operators give new interactions with other fields such as branching problems, automorphic functions or harmonic analysis on the homogeneous spaces $(G\times H)/\Delta(H)$ with $\Delta(H)\subseteq G\times H$ being the diagonal embedding of $H$. In fact, in \cite{MO12} the first and the last author already used the family of intertwining operators to derive the full branching law for the restriction of complementary series representations of $\SO(1,n)$ with respect to a symmetric pair. Further, the first and the second author combine in \cite{MO13} the explicit form of our integral kernels with a certain multiplicity-one property to derive estimates for the restriction of automorphic functions. Finally the integral kernels of our intertwining operators can be viewed as $H$-invariant distribution vectors on tensor product representations of $G\times H$ and therefore are related to harmonic analysis on the space $(G\times H)/\Delta(H)$ (see Remark~\ref{rem:InvDistVectors}).

We now describe our results in more detail.

\subsection{Symmetry breaking operators}

We restrict our attention to a certain subclass of representations $\pi$ and $\tau$. For this let $G$ be a real reductive Lie group in the Harish-Chandra class and $P=MAN$ a parabolic subgroup. Denote by $\fraka$ the Lie algebra of $A$ and write
\[ I^G(\nu)=\Ind_P^G(\1\otimes e^\nu\otimes\1) \]
for the spherical principal series representation of parameter $\nu\in\fraka_\CC^*$ (smooth normalized parabolic induction). Suppose that $(G, H)$ is a symmetric pair with corresponding involution $\sigma$, i.e.\ $H$ is open in $G^\sigma$. For a parabolic subgroup $P_H=M_HA_HN_H$ of $H$ we also consider the spherical principal series representation
\[ I^H(\nu')=\Ind_{P_H}^H(\1\otimes e^{\nu'}\otimes\1) \]
of parameter $\nu'\in(\fraka_H)_\CC^*$ where $\fraka_H$ is the Lie algebra of $A_H$. We then study Problem~\ref{prb:Intro1} for spherical principal series representations:
\begin{problem}
Construct explicitly \textit{symmetry breaking operators} in $\Hom_H(\pi|_H,\tau)$ for $\pi=I^G(\nu)$ and $\tau=I^H(\nu')$, and classify them.
\end{problem}

\subsection{The classical Knapp--Stein intertwiners}

For the case $H=G$ this problem has been well studied and \textit{symmetry breaking operators} are provided by the classical Knapp--Stein intertwiners. To obtain intertwiners between representations induced from the same parabolic subgroup $P=P_H$ we assume that
\begin{equation}
 \mbox{$P$ and its opposite $\overline{P}$ are conjugate via the Weyl group.}\label{eq:CondGIntro}\tag{G}
\end{equation}
Write $w_0$ for the longest Weyl group element and $\tilde{w}_0$ for one of its representatives so that \eqref{eq:CondGIntro} means $\tilde{w}_0^{-1}P\tilde{w}_0=\overline{P}=MA\overline{N}$. Then the operators
\[ \tilde{A}(\nu):I^G(\nu)\to I^G(w_0\nu), \quad \tilde{A}(\nu)f(g) = \int_{\overline{N}} f(g\tilde{w}_0\overline{n}) \td\overline{n} \]
belong to $\Hom_G(\pi,\tau)$ for $\pi=I^G(\nu)$, $\tau=I^G(w_0\nu)$ and sufficiently positive $\nu\in\fraka_\CC^*$. To extend the operators $\tilde{A}(\nu)$ meromorphically in $\nu$ we realize all the representations $I^G(\nu)$ on the space $C^\infty(X)$ with $X=K/(M\cap K)$ where $K$ is a maximal compact subgroup of $G$ whose Cartan involution leaves $MA$ invariant. Denote the corresponding $G$-action on $C^\infty(X)$ by $\pi_\nu$ so that $I^G(\nu)\cong(\pi_\nu,C^\infty(X))$. Using the $A$-projection $a:\overline{N}MAN\to A$ which is defined on the open dense subset $\overline{N}MAN\subseteq G$ we can write $\tilde{A}(\nu)$ as a singular integral operator $A(\nu):C^\infty(X)\to C^\infty(X)$ (cf. \cite[equation (7.37)]{Kna86}):
\begin{equation}
 A(\nu)f(k) = \int_K a(\tilde{w}_0^{-1}k^{-1}k')^{\nu-\rho}f(k') \td k',\label{eq:ClassicalKSasConvolution}
\end{equation}
where $\rho\in\fraka^*$ is half the sum of all positive roots of $(P,A)$. Then the operators $A(\nu)$ on $C^\infty(X)$ extend meromorphically in the parameter $\nu\in\fraka_\CC^*$.

\subsection{Invariant kernels for symmetric pairs}

In order to use the Knapp--Stein integral kernels in the construction of \textit{symmetry breaking operators} for more general symmetric pairs $(G,H)$ we assume in addition to \eqref{eq:CondGIntro} the following condition:
\begin{align}
 \mbox{$P$ is $\sigma$-stable,}\label{eq:CondHIntro}\tag{H}
\end{align}
which implies that $P_H:=P\cap H$ is a parabolic subgroup of $H$.
Various examples of triples $(G,H,P)$ satisfying conditions \eqref{eq:CondGIntro} and \eqref{eq:CondHIntro} are given in Section~\ref{sec:ExamplesCompatibleSymmetricPairs}. Among them are
\begin{itemize}
\item the rank one cases
\[ (G,H,P)=(\SU(1,n;\FF),\upS(\upU(1,m;\FF)\times\upU(n-m;\FF)),P_{\min}) \]
with $\FF=\RR,\CC,\HH$ and $0<m<n$ or $\FF=\OO$ and $n=m+1=2$, and $P_{\min}$ a suitable minimal parabolic subgroup,
\item the product cases
\[ (G,H,P)=(G'\times G',\Delta(G'),P'\times P'), \]
where $P'\subseteq G'$ is a parabolic subgroup which is conjugate to its opposite and $\Delta(G')\subseteq G'\times G'$ denotes the diagonal,
\item several examples $(G,H,P)$ with $P$ a maximal parabolic subgroup with abelian nilradical such as
\begin{multline*}
 \hspace{.9cm}(G,H,P)=(\GL(4n,\RR),\GL(2n,\CC),\\(\GL(2n,\RR)\times\GL(2n,\RR))\ltimes M(2n,\RR)).
\end{multline*}
\end{itemize}
%(We remark that in the first set of examples the group $\upO(1,n)$ is not connected and the group $\upU(1,n)$ is not semisimple but only reductive. However, our results have natural analogues for these cases and therefore we keep these groups to simplify our exposition.)

For $\alpha,\beta\in\fraka_\CC^*$ we define a kernel function
\[ K_{\alpha,\beta}(g,h) := a(\tilde{w}_0^{-1}g^{-1}h)^\alpha a(\tilde{w}_0^{-1}g^{-1}\sigma(g))^\beta, \qquad (g,h)\in G\times H, \]
whenever the expression makes sense. For $\beta=0$ this gives the integral kernel of the classical Knapp--Stein intertwiners in \eqref{eq:ClassicalKSasConvolution}. The domain of definition for the kernel $K_{\alpha,\beta}(g,h)$ is investigated in Section~\ref{sec:KernelDomain} where we prove that this domain is either empty or open dense in $G\times H$ and give a criterion to check this (see Proposition~\ref{prop:DDense}). In what follows we will simply assume that
\begin{equation}
 \mbox{the domain of definition for $K_{\alpha,\beta}(g,h)$ is non-empty.}\label{eq:CondDIntro}\tag{D}
\end{equation}
Condition~\eqref{eq:CondDIntro} is in particular satisfied in the above examples (see Corollary~\ref{cor:DDense}), but not in general (see Example~\ref{ex:DDenseForSp(n,R)}).

\subsection{Construction of symmetry breaking operators}

Analogously as for $G$, we realize the representations $I^H(\nu')$ of $H$ on the space $C^\infty(X_H)$ with $X_H=K_H/(M_H\cap K_H)$ and denote the $H$-action on this space by $\tau_{\nu'}$. Further, denote by $\rho'\in\fraka_H^*$ half the sum of all positive roots of $(P_H,A_H)$. Assuming conditions \eqref{eq:CondGIntro}, \eqref{eq:CondHIntro} and \eqref{eq:CondDIntro} we prove in Theorem~\ref{thm:ConvergenceMeromorphicContinuationIntertwiners}:

\begin{thmalph}\label{thm:IntroA}
\begin{enumerate}[(1)]
\item For $f\in C^\infty(X)$ the integral
\[ A(\alpha,\beta)f(k_H) := \int_K K_{\alpha,\beta}(k,k_H)f(k) \td k, \qquad k_H\in K_H, \]
converges absolutely for $\Re\alpha,\Re\beta$ in an open set $\fraka_+^*\subseteq\fraka^*$ (see \eqref{eq:Defa^*_+} for the precise definition) and extends meromorphically in $\alpha,\beta\in\fraka_\CC^*$ to a non-trivial family of continuous linear operators $A(\alpha,\beta):C^\infty(X)\to C^\infty(X_H)$.
\item For a regular point $(\alpha,\beta)$ of $A(\alpha,\beta)$ and
\begin{equation}
 \nu=-w_0\alpha+\sigma\beta-w_0\beta+\rho, \qquad \nu'=-\alpha|_{\fraka^\sigma_\CC}-\rho'\label{eq:IntroParameterRelation}
\end{equation}
the map $A(\alpha,\beta)$ defines an $H$-intertwining operator $\pi_\nu|_H\to\tau_{\nu'}$, i.e.
$$ A(\alpha,\beta)\in\Hom_H(\pi_\nu|_H,\tau_{\nu'}). $$
\end{enumerate}
\end{thmalph}

The relation \eqref{eq:IntroParameterRelation} between the parameters $\alpha,\beta$ of the kernel $K_{\alpha,\beta}(g,h)$ and the induction parameters $\nu,\nu'$ is discussed in Section~\ref{sec:InductionParameters}. In the case where $w_0|_\fraka=-1$ the mapping $(\alpha,\beta)\mapsto(\nu,\nu')$ can be turned into a bijection onto a certain subset of induction parameters.

We remark that the intertwining operators $A(\alpha,\beta)$ are known in two special cases:
\begin{itemize}
\item For $(G,H,P)=(G'\times G',\Delta(G'),P'\times P')$ the operators $A(\alpha,\beta)$ are $G'$-intertwining operators $I^{G'}(\nu_1)\otimes I^{G'}(\nu_2)\to I^{G'}(\nu_3)$ (see Example~\ref{ex:KernelExamples}~(2)). These operators correspond to invariant trilinear forms $I^{G'}(\nu_1)\times I^{G'}(\nu_2)\times I^{G'}(-\nu_3)\to\CC$ which were investigated for various groups ${G'}$, see \cite{BC12,BSKZ,BR04,CKOP11,CO11,Dei06,Oks73}.
\item For $(G,H,P)=(\SO(1,n),\SO(1,n-1),P_{\min})$ the operators $A(\alpha,\beta)$ were previously investigated by Kobayashi--Speh~\cite{KS,KS14} (see also \cite{Kob13,MO12}). They use these operators to completely determine the spaces $\Hom_H(\pi|_H,\tau)$ for $\pi=\pi_\nu$, $\tau=\tau_{\nu'}$ and arbitrary $\nu,\nu'$.
\end{itemize}

\subsection{The space $\Hom_H(\pi|_H,\tau)$}

Since the kernel $K_{\alpha,\beta}(g,h)$ is left-invariant under the diagonal action of $H$ and right-equivariant under the action of $P\times P_H$, it can be viewed as an $H$-invariant section of a certain line bundle over the double flag variety $G/P\times H/P_H$. This suggests a connection between $\dim\Hom_H(\pi|_H,\tau)$ and the number of open $H$-orbits in $G/P\times H/P_H$. In fact, in Section~\ref{sec:Uniqueness} we outline a general technique to relate these two quantities, and apply this technique in Theorem~\ref{thm:UniquenessRank1} and \ref{thm:UniquenessNewExample} in two special cases. We further investigate the number of open $H$-orbits in $G/P\times H/P_H$ in Proposition~\ref{prop:OrbitsDoubleFlag}. The results can be summarized as follows:

\begin{thmalph}\label{thm:IntroB}
\begin{enumerate}[(1)]
\item The number of open $H$-orbits in $G/P\times H/P_H$ equals the number of open $M_HA_H$-orbits in $\frakn^{-\sigma}$ where $\frakn$ is the Lie algebra of $N$ and
\[ \frakn^{-\sigma} = \{X\in\frakn:\sigma X=-X\}. \]
\item For $(G,H)=(\SU(1,n;\FF),\upS(\upU(1,m;\FF)\times\upU(n-m;\FF)))$, $\FF=\RR,\CC,\HH$ and $0<m<n$ or $\FF=\OO$ and $n=m+1=2$, there is only one open $H$-orbit in $G/P\times H/P_H$ and
\[ \Hom_H(\pi_\nu|_H,\tau_{\nu'}) = \CC\cdot A(\alpha,\beta) \]
for generic parameters (see Theorem~\ref{thm:UniquenessRank1} for the precise statement).
\item For $(G,H)=(\GL(4n,\RR),\GL(2n,\CC))$ with parabolic subgroups $P=(\GL(2n,\RR)\times\GL(2n,\RR))\ltimes M(2n,\RR)$ and $P_H=(\GL(n,\CC)\times\GL(n,\CC))\ltimes M(n,\CC)$, there is only one open $H$-orbit in $G/P\times H/P_H$ and
\[ \Hom_H(\pi_\nu|_H,\tau_{\nu'}) = \CC\cdot A(\alpha,\beta) \]
for generic parameters that agree on the center of $\frakg$ (see Theorem~\ref{thm:UniquenessNewExample} for the precise statement).
\end{enumerate}
\end{thmalph}

We remark that Kobayashi--Oshima~\cite{KO11} showed $\dim\Hom_H(\pi|_H,\tau)<\infty$ for all irreducible admissible representations $\pi$ and $\tau$ if $H$ has an open orbit on $G/P_{\min}\times H/P_{H,\min}$ for $P_{\min}$ and $P_{H,\min}$ minimal parabolic subgroups. This condition is stronger than $H$ having an open orbit on $G/P\times H/P_H$. For $\FF=\RR,\CC$ and $m=n-1$, Theorem~\ref{thm:IntroB}~(2) also follows from the multiplicity-one theorem by Sun--Zhu~\cite{SZ12}.

\subsection{Outlook}

We indicate some possible further lines of research:
\begin{itemize}
\item (Singular integral operators) In the non-compact realizations of $I^G(\nu)$ and $I^H(\nu')$ on functions on $\overline{N}$ and $\overline{N}_H$ the intertwiners $A(\alpha,\beta)$ are singular integral operators on nilpotent Lie groups. The meromorphic nature of these operators from a viewpoint of classical analysis was studied in detail by Kobayashi--Speh~\cite{KS,KS14} for the case $(G,H)=(\upO(1,n),\upO(1,n-1))$ and is of interest for other cases, too.
\item (Bernstein--Sato identities) Our proof of meromorphic extension in the parameters $\alpha,\beta$ does not provide any information about the location of the poles and the residues of $A(\alpha,\beta)$. In Section~\ref{sec:BernsteinSatoIdentities} we outline a method due to Beckmann--Clerc~\cite{BC12} to obtain explicit Bernstein--Sato identities for the kernel function $K_{\alpha,\beta}(g,h)$ which can be used to study this problem. We expect this method to work at least for some subclasses of groups such as rank one groups or groups with maximal parabolic subgroups having abelian nilradical.
\item (Uniqueness) In Section~\ref{sec:Uniqueness} we describe a strategy to prove generic bounds for $\dim\Hom_H(\pi|_H,\tau)$ for $\pi=I^G(\nu)$ and $\tau=I^H(\nu')$. This strategy is applied in Section~\ref{sec:RankOne} to prove the uniqueness result in Theorem~\ref{thm:IntroB}~(2) and is expected to work also in other cases where $H$ has an open orbit on $G/P\times H/P_H$.
\item (Branching laws) For $(G,H)=(\SO(1,n),\SO(1,m)\times\SO(n-m))$ the first and the third author use the operators $A(\alpha,\beta)$ in \cite{MO12} to find the full branching law for the restriction of spherical complementary series of $G$ to $H$. The operators $A(\alpha,\beta)$ might also shed some light on branching problems for other symmetric pairs $(G,H)$.
\item (Automorphic functions) Using the multiplicity-one statement in Theorem~\ref{thm:IntroB} and evaluating the intertwining operators $A(\alpha,\beta)$ explicitly at the spherical vector the first and the second author derive estimates for the restriction of automorphic functions on real hyperbolic manifolds in \cite{MO13}. This technique due to Bernstein--Reznikov~\cite{BR04} is expected to work also for other locally symmetric spaces using our explicit intertwining operators.\\
\end{itemize}

\emph{Acknowledgement:} We thank T. Kobayashi for helpful discussions on the topic of this paper.
\section{Parabolic subgroups and the double flag variety}

We fix the setting and recall some basic structure theory of reductive groups and symmetric pairs. Further we investigate the orbit structure on double flag varieties and give various examples.

\subsection{Parabolic subgroups and decompositions}

Let $G$ be a real reductive Lie group in the Harish-Chandra class (see e.g. \cite[Chapter VII]{Kna02} for details). Let $\theta$ be a Cartan involution and $K=G^\theta$ the corresponding maximal compact subgroup. Write
\begin{equation*}
 \frakg = \frakk+\fraks
\end{equation*}
for the corresponding Cartan decomposition of the Lie algebra $\frakg$ of $G$. Let $\langle\blank,\blank\rangle$ denote a non-degenerate invariant form of $\frakg$ which is negative definite on $\frakk$ and positive definite on $\fraks$.

We fix a minimal parabolic subgroup $P_{\min}=M_{\min}A_{\min}N_{\min}$ of $G$ with $\theta$-stable Levi subgroup $M_{\min}A_{\min}$. Denote by $\overline{P}_{\min}=\theta(P_{\min})=M_{\min}A_{\min}\overline{N}_{\min}$ its opposite parabolic subgroup, $\overline{N}_{\min}=\theta(N_{\min})$. Write $\frakm_{\min}$, $\fraka_{\min}$, $\frakn_{\min}$ and $\overline{\frakn}_{\min}$ for the Lie algebras of $M_{\min}$, $A_{\min}$, $N_{\min}$ and $\overline{N}_{\min}$, respectively. Then $\fraka_{\min}\subseteq\fraks$ is a maximal abelian subalgebra and $M_{\min}=Z_K(\fraka_{\min})$. Denote the root system of the pair $(\frakg,\fraka_{\min})$ by $\Sigma=\Sigma(\frakg,\fraka_{\min})$ and let $\Sigma^+=\Sigma^+(\frakg,\fraka_{\min})$ be the subset of roots in $\frakn_{\min}$. The corresponding set of simple roots will be denoted by $\Pi=\Pi(\frakg,\fraka_{\min})$.

The finite group $W=N_K(\fraka_{\min})/Z_K(\fraka_{\min})$ is identified with the Weyl group of the root system $\Sigma$. For every $w\in W$ we choose a representative $\tilde{w}\in N_K(\fraka_{\min})$. Denote by $w_0\in W$ the longest element in $W$. Since the longest element in $W$ is unique we have $w_0^{-1}=w_0$. Therefore $w_0^2=1$ which implies $\tilde{w}_0^2\in M_{\min}$. Since $\Ad(\tilde{w}_0)$ maps $\Sigma^+$ to $(-\Sigma^+)$ we further have $\tilde{w}_0P_{\min}\tilde{w}_0^{-1}=\tilde{w}_0^{-1}P_{\min}\tilde{w}_0=\overline{P}_{\min}$. We write the $W$-action on $\fraka_{\min}$, its dual $\fraka_{\min}^*$ and $A_{\min}$
 as $wH$, $w\lambda$ and ${^{w}\!a}$, respectively ($w\in W$, $H\in \fraka_{\min}$, $\lambda\in\fraka_{\min}^*$, $a\in A_{\min}$).

For each $\alpha\in\Pi$ let $H_\alpha\in\fraka_{\min}$ such that
\[ \langle H_\alpha,H\rangle = \alpha(H), \qquad H\in\fraka_{\min}. \]
Then $(H_\alpha)_{\alpha\in \Pi}$ forms a basis of $\fraka_{\min}$.

The standard parabolic subgroups $P=MAN$ of $G$ containing $P_{\min}$ correspond to the subsets $F\subseteq\Pi$ in the following sense: $P$ is the normalizer of its Lie algebra $\frakp=\frakm+\fraka+\frakn$ where
\[ \fraka = \{H\in\fraka_{\min}:\alpha(H)=0\ \forall\,\alpha\in F\} \]
and
\[ \frakm = \frakm_{\min}\oplus\bigoplus_{\alpha\in F}\RR H_\alpha\oplus\bigoplus_{\substack{\alpha\in\Sigma\\\alpha|_\fraka=0}}\frakg_\alpha, \qquad \qquad \frakn = \bigoplus_{\substack{\alpha\in\Sigma^+\\\alpha|_\fraka\neq0}}\frakg_\alpha. \]
We clearly have $M_{\min}\subseteq M$, $A\subseteq A_{\min}$ and $N\subseteq N_{\min}$. Note that
\begin{equation}
 \fraka_{\min} = \fraka \oplus (\fraka_{\min}\cap\frakm).\label{eq:DecompAmin}
\end{equation}
Put $\overline{N}:=\theta(N)$ and $\overline{\frakn}:=\theta(\frakn)$. Then $\overline{N}MAN$ is an open dense subset of $G$ and we have the decomposition $G=KMAN$.

Let $W_P:=Z_W(\fraka)$, the centralizer of $\fraka$ in $W$. Then we have the generalized Bruhat decomposition (see e.g.\ \cite[Proposition 1.2.1.10]{War72})
\begin{equation}
 G = \bigsqcup_{[w]\in W_P\backslash W/W_P} P\tilde{w}P.\label{eq:BruhatDecomp}
\end{equation}
The Bruhat cell $P\tilde{w}_0P$ with $w_0\in W$ the longest Weyl group element is open dense since $P\tilde{w}_0P\supseteq P_{\min}\tilde{w}_0P_{\min}=\tilde{w}_0\overline{N}_{\min}M_{\min}A_{\min}N_{\min}$.

The parabolic subgroup opposite to $P$ is given by $\overline{P}=MA\overline{N}$. In what follows we will assume the following condition:
\begin{equation}
 \mbox{$P$ and $\overline{P}$ are conjugate via the Weyl group,}\label{eq:CondG}\tag{G}
\end{equation}
i.e.\ for the longest Weyl group element $w_0\in W$ we have $\tilde{w}_0^{-1}P\tilde{w}_0=\overline{P}$. This implies $\tilde{w}_0^{-1}M\tilde{w}_0=M$, $\tilde{w}_0^{-1}A\tilde{w}_0=A$ and $\tilde{w}_0^{-1}N\tilde{w}_0=\overline{N}$. Hence the decomposition \eqref{eq:DecompAmin} is stable under $\Ad(\tilde{w}_0)$. We further note that under Condition~\eqref{eq:CondG} the open dense subset $\overline{N}P\subseteq G$ is up to multiplication with $\tilde{w}_0$ the cell $P\tilde{w}_0P=\tilde{w}_0\overline{P}P=\tilde{w}_0\overline{N}P$ in the Bruhat decomposition~\eqref{eq:BruhatDecomp}.

\begin{remark}
Note that for $P=P_{\min}$ we always have $\tilde{w}_0^{-1}P\tilde{w}_0=\overline{P}$. Further, we have the following implications where for every implication $\Longrightarrow$ the converse statement is not true:
\begin{align*}
 \mbox{$\Sigma$ is not of type $A_n$ ($n\geq2$), $D_{2n+1}$ ($n\geq1$) or $E_6$} \Longleftrightarrow{}&
 w_0=-\id \text{ on $\fraka_{\min}$}\\
 \Longrightarrow{}& w_0=-\id \text{ on $\fraka$}\\
 \Longrightarrow{}& \tilde{w}_0^{-1}P\tilde{w}_0=\overline{P}.
\end{align*}
For example, for $G=\SL(2n,\RR)$, $n\geq2$, with parabolic subgroup corresponding to $MA=\upS(\GL(n,\RR)\times\GL(n,\RR))$ we have $w_0=-\id$ on $\fraka$ but $w_0\neq-\id$ on $\fraka_{\min}$ since $\Sigma$ is of type $A_{2n-1}$ here. Further, for $G=\SL(n,\RR)$, $n\geq3$, with $P=P_{\min}$ we have $w_0\neq-\id$ on $\fraka$, but $\tilde{w}_0^{-1}P\tilde{w}_0=\overline{P}$.
\end{remark}

Corresponding to the decomposition $G=KMAN$ we write
\begin{equation*}
 g \in \kappa(g)Me^{H(g)}N \subseteq KMAN,
\end{equation*}
where $\kappa(g)\in K$ and $H(g)\in\fraka$. Note that $\kappa(g)$ is only determined up to multiplication by $M\cap K$ from the right. Anytime we use $\kappa(g)$, however, the expression will be independent of the different choices. 

For $g\in\overline{N}MAN$ we further write
\begin{equation*}
 g \in \overline{N}Ma(g)N \subseteq \overline{N}MAN,
\end{equation*}
where $a(g)\in A$. Then the function $a:\overline{N}MAN\to A$ satisfies
\begin{equation}
 a(m'a'\overline{n}gman) = a'a(g)a, \qquad m,m'\in M,\,a,a'\in A,n\in N,\overline{n}\in\overline{N}.\label{eq:EquivarianceAFunction}
\end{equation}

\begin{remark}\label{rem:PropertiesAFct1}
\begin{enumerate}
\item Since $\tilde{w}_0^2\in M_{\min}\subseteq M$ we have $a(\tilde{w}_0^{-1}g)=a(\tilde{w}_0g)$ for all $g\in\tilde{w}_0\overline{N}MAN$.
\item For $g=\tilde{w}_0\overline{n}man\in \tilde{w}_0\overline{N}MAN$ we have
\[ g^{-1}=n^{-1}a^{-1}m^{-1}\overline{n}^{-1}\tilde{w}_0^{-1}=\tilde{w}_0^{-1}\overline{n}'m'({^{w_0}\!a}^{-1})n'\in \tilde{w}_0^{-1}\overline{N}MAN. \]
Hence
\[ a(\tilde{w}_0^{-1}g^{-1}) = a(\tilde{w}_0g^{-1}) = {^{w_0}\!a}(\tilde{w}_0^{-1}g)^{-1}. \]
In the case where $w_0=-\id$ on $\fraka$ this yields
\begin{equation}
 a(\tilde{w}_0^{-1}g) = a(\tilde{w}_0^{-1}g^{-1}).\label{eq:aFctInvUnderInversion}
\end{equation}
\end{enumerate}
\end{remark}

Corresponding to these decompositions we recall two important integral formulas. For this let $\td k$ be the normalized Haar measure on $K$. Then the Haar measure $\td\overline{n}$ on $\overline{N}$ can be normalized such that for $f\in L^1(K)$ which is right-invariant under $K\cap M$ we have (see \cite[equation (7.4)]{Kna86}):
\begin{equation}
 \int_K f(k)\td k = \int_K f(\kappa(g^{-1}k))e^{-2\rho H(g^{-1}k)}\td k \qquad \forall\,g\in G.\label{eq:IntFormulaGActionOnK}
\end{equation}
Further, for all $K\cap M$-right-invariant functions $f\in L^1(K)$ we have (see \cite[equation (5.25)]{Kna86}):
\begin{align}
 \int_K f(k)\td k = \int_{\overline{N}}f(\kappa(\overline{n}))e^{-2\rho H(\overline{n})}\td\overline{n}.\label{eq:IntFormulaKNbar}
\end{align}

\subsection{The function $a^\lambda$}

For $a\in A$ and $\lambda\in\fraka_\CC^*$ we write
\begin{equation*}
 a^\lambda := e^{\lambda(\log(a))}.
\end{equation*}
This defines a function $a^\lambda:\overline{N}MAN\to\CC,\,g\mapsto a(g)^\lambda$ for every $\lambda\in\fraka_\CC^*$. We study the behaviour of these functions near the possible singularities $G\setminus\overline{N}MAN$.

Since the restriction of $\langle\blank,\blank\rangle$ to $\fraka_{\min}$ defines an inner product on $\fraka_{\min}$, it identifies $\fraka_{\min}^*\cong\fraka_{\min}$ and in turn also defines an inner product on $\fraka_{\min}^*$.
%Extend $\fraka_{\min}$ to a Cartan subalgebra $\frakh$ of $\frakg$ which necessarily satisfies $\frakh=(\frakk\cap\frakh)\oplus\fraka_{\min}$. By the Cartan--Helgason Theorem \cite[V\,\S4, Theorem~4.1]{Hel84} the highest weights of finite-dimensional representations of $G$ containing a $K$-fixed vector are those $\lambda\in\frakh_\CC^*$ vanishing on $\frakk\cap\frakh$ such that the restriction $\lambda|_{\fraka_{\min}}$ is contained in the set
We define 
\begin{equation*}
 \Lambda^+(\fraka_{\min}) := \{\lambda\in\fraka_{\min}^*,\,\tfrac{\langle\lambda,\alpha\rangle}{\langle\alpha,\alpha\rangle}\in\NN_0\,\forall\,\alpha\in\Sigma^+(\frakg,\fraka_{\min})\}.
\end{equation*}
%We identify elements of $\Lambda^+(\fraka_{\min})$ with the corresponding weights in $\frakh_\CC^*$. 
The set $\Lambda^+(\fraka_{\min})$ contains a basis of $\fraka_{\min}^*$ and hence $\RR\,\textup{-\,span}\, \Lambda^+(\fraka_{\min})=\fraka_{\min}^*$. In view of the decomposition~\eqref{eq:DecompAmin} we put
\begin{align}
 \Lambda^+(\fraka) &:= \{\lambda|_\fraka:\lambda\in\Lambda^+(\fraka_{\min}),\,\lambda|_{\frakm\cap\fraka_{\min}}=0\},\label{eq:DefLambdaG+a}\\
 \fraka^*_+ &:= \Rspan\,\Lambda^+(\fraka).\label{eq:Defa^*_+}
\end{align}
Using \eqref{eq:DecompAmin} we view elements of $\fraka^*$ and hence of $\Lambda^+(\fraka)$ as functionals on $\fraka_{\min}$ which vanish on $\frakm\cap\fraka_{\min}$.

\begin{lemma}\label{lem:BoundednessAFct}
\begin{enumerate}[(1)]
\item For $\lambda\in\Lambda^+(\fraka)$ the function $a^\lambda$ is a matrix coefficient of an irreducible finite-dimensional representation of $G$ and hence extends to a real-analytic function on $G$.
\item For $\lambda\in\fraka^*_+$ the function $a^\lambda$ is bounded on the open dense subset $K\cap\overline{N}MAN\subseteq K$.
\item $\RR\,\textup{-\,span}\, \Lambda^+(\fraka)=\fraka^*$.
\item For each $\lambda$ in the open set
\[ \fraka^*_{+,\reg} := \{\lambda\in\fraka^*:\lambda(H_\alpha)>0\,\forall\,\alpha\in\Pi\setminus F\} \]
we have $w\cdot\lambda\neq\lambda$ for all $w\in W\setminus W_P$. Further, $\fraka^*_{+,\reg}\cap\Lambda^+(\fraka)\neq\emptyset$.
\end{enumerate}
\end{lemma}

\begin{proof}
\begin{enumerate}
\item Let $\lambda\in\Lambda^+(\fraka_{\min})$ with $\lambda|_{\frakm\cap\fraka_{\min}}=0$.  Write $a_{\min}^\lambda$ for the $a$-function of $P_{\min}$, namely $a_{\min}(g)^\lambda=a^\lambda$ for $g\in\overline{N}_{\min}M_{\min}aN_{\min}$ and $a\in A_{\min}$.  Then the decompositions $N_{\min}=N(N_{\min}\cap M)$, $A_{\min}=A(A_{\min}\cap M)$ and $\overline{N}_{\min}=\overline{N}(\overline{N}_{\min}\cap M)$ imply that
$$ a(g)^\lambda=a_{\min}(g)^\lambda \qquad \mbox{for $g\in \overline{N}_{\min}M_{\min}A_{\min}N_{\min}$.} $$
Therefore it is enough to show the claim for $a_{\min}(g)^\lambda$.\\
Let $G_{ss}$ be the connected subgroup of $G$ with Lie algebra $[\frakg,\frakg]$.  By the Cartan--Helgason Theorem \cite[V\,\S4, Theorem~4.1]{Hel84}, the function $a_{\min}^\lambda|_{G_{ss}\cap \overline{N}_{\min}M_{\min}A_{\min}N_{\min}}$ extends to a matrix coefficient of a finite-dimensional irreducible representation of $G_{ss}$ with a $(K\cap G_{ss})$-fixed vector.  This implies that the right $G_{ss}$-translates of $a_{\min}^\lambda$, i.e.\ $a_{\min}(\blank g)^\lambda$ for $g\in G_{ss}$, span an irreducible finite-dimensional representation of $G_{ss}$, which we denote by $(\pi,V)$.  On the other hand, $a_{\min}^\lambda$ is right invariant by $M_{\min}$ and transforms by a character under the right action of $Z_G$, the center of $G$.  Since $M_{\min}$ meets every connected component of $G$ (see e.g.\ \cite[Proposition 7.33]{Kna02}), we have $G=G_{ss}M_{\min}Z_G$ and hence $V$ is stable under $G$.  Therefore, $\pi$ extends to a representation of $G$.  The representation $V$ of $G$ has a highest weight vector $\phi=a_{\min}^\lambda\in V$ with weight $\lambda$.
Define $\phi^*\in V^*$ by $\phi^*(f)=f(e)$ for $f\in V$.  Then $\phi^*$ is a lowest weight vector in the contragredient representation $V^*$ with weight $-\lambda$ and we have $a_{\min}(g)^\lambda=(\pi(g)\phi|\phi^*)$.
\item That $K\cap\overline{N}MAN\subseteq K$ is open dense follows immediately from the fact that $\overline{N}MAN$ is open dense in $G$ and the isomorphism $G/P\cong K/(K\cap M)$. The boundedness is then clear by (1).
\item We may assume that $\frak{g}$ is semisimple.  Let $\lambda=\sum_{\alpha\in\Pi}\lambda_\alpha\alpha\in\fraka_{\min}^*$ with $\lambda_\alpha\in2\ZZ$. Since $\Sigma$ is a root system we have
$$ \frac{\langle\lambda,\beta\rangle}{\langle\beta,\beta\rangle}\in\ZZ \quad \forall\,\beta\in\Sigma. $$
It follows that
\begin{align*}
 \lambda\in\Lambda^+(\fraka_{\min}) \quad &\Leftrightarrow \quad \langle\lambda,\beta\rangle\geq0 && \forall\,\beta\in\Sigma^+\\
 &\Leftrightarrow \quad \langle\lambda,\beta\rangle\geq0 && \forall\,\beta\in\Pi\\
 &\Leftrightarrow \quad \sum_{\alpha\in\Pi}\lambda_\alpha\frac{2\langle\alpha,\beta\rangle}{\langle\beta,\beta\rangle}\geq0 && \forall\,\beta\in\Pi.
\intertext{Moreover}
 \lambda|_{\frakm\cap\fraka_{\min}}=0\quad &\Leftrightarrow \quad \lambda(H_\beta)=0 && \forall\,\beta\in F\\
 &\Leftrightarrow \quad \sum_{\alpha\in\Pi}\lambda_\alpha\langle\alpha,\beta\rangle=0 && \forall\,\beta\in F\\
  &\Leftrightarrow \quad \sum_{\alpha\in\Pi}\lambda_\alpha\frac{2\langle\alpha,\beta\rangle}{\langle\beta,\beta\rangle}=0 && \forall\,\beta\in F.
\end{align*}
The coefficients $A_{\alpha\beta}=\frac{2\langle\alpha,\beta\rangle}{|\beta|^2}$ are the entries of the Cartan matrix $A$ of the root system $\Sigma$ and we can write
$$ \lambda\in\Lambda^+(\fraka) \quad \Leftrightarrow \quad \sum_{\alpha\in\Pi}A_{\alpha,\beta}\lambda_\alpha \begin{cases}\geq0 & \mbox{for $\beta\in\Pi\setminus F$,}\\=0 & \mbox{for $\beta\in F$.}\end{cases} $$
Since the matrix $A$ is invertible and has integer entries it follows that there exists a basis of $\fraka^*$ consisting of elements in $\Lambda^+(\fraka)$ and the claims follows.
\item Let $\lambda\in\fraka_{\min}^*$ with $\lambda(H_\alpha)=0$ for all $\alpha\in F$ and $\lambda(H_\alpha)>0$ for all $\alpha\in\Pi\setminus F$. Hence $\lambda$ is in the closure of the positive Weyl chamber. Assume $w\cdot\lambda=\lambda$ for an element $w\in W$ then by \cite[Lemma B in 10.3]{Hum78} we find $w=w_1\cdots w_s$ with $w_j$ simple reflections leaving $\lambda$ invariant. A simple reflection $w_j$ along $\alpha_j\in\Pi$ leaving $\lambda$ invariant satisfies
\[ \lambda(H_{\alpha_j}) = (w_j\cdot\lambda)(H_{\alpha_j}) = \lambda(w_j^{-1}H_{\alpha_j}) = -\lambda(H_{\alpha_j}). \]
Hence $\lambda(H_{\alpha_j})=0$ and therefore $\alpha_j\in F$. But this means that $w_j\in W_P$ for all $j$ whence $w\in W_P$. The fact that $\fraka^*_{+,\reg}\cap\Lambda^+(\fraka)\neq\emptyset$ follows from (3).\qedhere
\end{enumerate}
\end{proof}

\begin{remark}
For $P$ a minimal parabolic subgroup similar results as in Lemma~\ref{lem:BoundednessAFct} were proved in \cite[Lemma 5.1]{CKOP11}.
\end{remark}

For $\lambda\in\fraka_\CC^*$ consider the function $g\mapsto a(\tilde{w}_0^{-1}g)^\lambda$. By Remark~\ref{rem:PropertiesAFct1} we find that
\begin{equation}
 a(\tilde{w}_0^{-1}g^{-1})^\lambda = a(\tilde{w}_0^{-1}g)^{-w_0\lambda}.\label{eq:AlambdaFctInverse}
\end{equation}
Note that $w_0\Lambda^+(\fraka_{\min})=-\Lambda^+(\fraka_{\min})$ since $w_0\Sigma^+=-\Sigma^+$. Hence $w_0\Lambda^+(\fraka)=-\Lambda^+(\fraka)$.

\begin{lemma}\label{lem:MatrixCoeffPolynomialOnNbar}
For $\lambda\in\Lambda^+(\fraka)$ the function
\[ p^\lambda:\overline{\frakn}\to\CC, \quad X\mapsto a(\tilde{w}_0^{-1}e^X)^\lambda \]
is a polynomial on $\overline{\frakn}$. It has the following properties:
\begin{enumerate}[(1)]
\item (Homogeneity) For $a\in A$ we have $p^\lambda(\Ad(a)X)=a^{w_0\lambda-\lambda}p^\lambda(X)$, $X\in\overline{\frakn}$.
\item (Parity) We have $p^\lambda(-X)=p^{-w_0\lambda}(X)$, $X\in\overline{\frakn}$.
\item (Zero set) For $\lambda\in\fraka_{+,\reg}^*\cap\Lambda^+(\fraka)$ we have $e^X\in N\tilde{w}_0MAN$ if and only if $p^\lambda(X)\neq0$.
\end{enumerate}
\end{lemma}

\begin{proof}
As seen in the proof of Lemma~\ref{lem:BoundednessAFct}~(1) the function $a^\lambda$ is the matrix coefficient $a(g)^\lambda=(\pi(g)\phi|\phi^*)$ of a finite-dimensional representation $(\pi, V)$ of $G$ with highest weight $\lambda\in\Lambda^+(\fraka_{\min})$, $\lambda|_{\fraka_{\min}\cap\frakm}=0$, $\phi$ a highest weight vector in $V$ and $\phi^*$ a lowest weight vector in $V^*$. Hence $p^\lambda(X)=(e^{\td\pi(X)}\phi|\pi(\tilde{w}_0)^*\phi^*)$. Since $\frakn$ acts nilpotently on $V$, the map $\frakn\to V$, $X\mapsto e^{\td\pi(X)}\phi$ is a polynomial and so is $p^\lambda$.
We now prove properties (1), (2) and (3):
\begin{enumerate}[(1)]
\item For $a\in A$ we have by \eqref{eq:EquivarianceAFunction}
\begin{equation*}
 p^\lambda(\Ad(a)X) = a({^{w_0}\!a}\tilde{w}_0^{-1}e^Xa^{-1})^\lambda = a^{w_0\lambda-\lambda}p^\lambda(X).
\end{equation*}
\item With \eqref{eq:AlambdaFctInverse} we find
\begin{equation*}
 p^{\lambda}(-X) = a(\tilde{w}_0^{-1}e^{-X})^\lambda = a(\tilde{w}_0^{-1}e^X)^{-w_0\lambda} = p^{-w_0\lambda}(X).
\end{equation*}
\item Certainly $p^\lambda(X)\neq0$ if $e^X\in N\tilde{w}_0MAN$ since the function $a^\lambda$ is positive on $\overline{N}MAN$. Now assume $e^X\notin N\tilde{w}_0MAN$. Then $e^X$ must be contained in some other Bruhat cell $P\tilde{w}P$, $W_PwW_P\neq W_Pw_0W_P$, whence $\tilde{w}_0^{-1}e^X\in\overline{P}\tilde{w}_0^{-1}\tilde{w}P$. Write $\tilde{w}_0^{-1}e^X=ma\overline{n}\tilde{w}_0^{-1}\tilde{w}m'a'n'$ then
\begin{equation*}
 p^\lambda(X) = (aa')^\lambda(\pi(\tilde{w}_0^{-1}\tilde{w})\phi|\phi^*).
\end{equation*}
But $\pi(\tilde{w}_0^{-1}\tilde{w})\phi$ lies in the weight space of weight $(\tilde{w}_0^{-1}\tilde{w})\cdot\lambda$ which is different from $\lambda$ by Lemma~\ref{lem:BoundednessAFct}~(4). Hence $(\pi(\tilde{w}_0^{-1}\tilde{w})\phi|\phi^*)=0$ which implies $p^\lambda(X)=0$.\qedhere
\end{enumerate}
\end{proof}

\subsection{Compatible symmetric pairs}

Fix a $\theta$-stable symmetric pair $(G, H)$ with corresponding involution $\sigma$, i.e.\ $G^\sigma_0\subseteq H\subseteq G^\sigma$ and $\theta\sigma=\sigma\theta$. We assume that
\begin{equation}
 \mbox{$P$ is $\sigma$-stable.}\label{eq:CondH}\tag{H}
\end{equation}
Then $\sigma$ stabilizes $M$, $A$ and $N$ and $P_H:=P\cap H$ is a parabolic subgroup of $H$. Denote by $P_H=M_HA_HN_H$ its Langlands decomposition and by $\frakp_H=\frakm_H+\fraka_H+\frakn_H$ the corresponding Lie algebras. Then $M_HA_H\subseteq MA$ and $N_H\subseteq N$. Hence $M_H\subseteq M$ and $A_H=(M\cap A_H)(A\cap A_H)$. For the Lie algebra $\fraka_H$ of $A_H$ this means that
\[ \fraka_H=(\frakm\cap\fraka_H)\oplus(\fraka\cap\fraka_H). \]
Replacing $P_{\min}$ by some conjugation by $M$, we may and do assume that
\begin{equation}
 \sigma(\fraka_{\min})=\fraka_{\min}.\label{eq:SigmaOnAmin}
\end{equation}

For $a\in A_H$ with $a=a_Ma_A$, $a_M\in M\cap A_H$, $a_A\in A\cap A_H$, and $\lambda\in\fraka_{\CC}^*$ we denote by $a^\lambda$ the part $a_A^\lambda$. In this notation we have for $g\in G$ and $man\in M_HA_HN_H$:
\begin{equation*}
 a(gman)^\lambda = a(g)^\lambda a^\lambda.
\end{equation*}

\begin{lemma}\label{lem:SigmaW0Commute}
We have $\sigma(\tilde{w}_0)\tilde{w}_0\in M$. In particular $\sigma$ and $w_0$ commute on $\fraka$.
\end{lemma}

\begin{proof}
The assumptions \eqref{eq:CondG} and \eqref{eq:CondH} together with the property \eqref{eq:SigmaOnAmin} imply that $(\sigma w_0)^2 \Sigma(\frakn,\fraka_{\min})=\Sigma(\frakn,\fraka_{\min})$.  Hence there exists an element $w_M$ in the Weyl group of $M$ such that $w_M (\sigma w_0)^2 \Sigma^+=\Sigma^+$.  Since a Weyl group element which stabilizes the set of positive roots must be the identity, we have $\tilde{w}_M \sigma(\tilde{w}_0)\tilde{w}_0\in Z_K(\fraka_{\min})=M_{\min}$. Now $\tilde{w}_M\in M$ and $M_{\min}\subseteq M$ and hence the first claim follows.
For the second claim note that $M$ centralizes $\fraka$ and hence $(\sigma w_0)^2$ acts as the identity on $\fraka$. Since $\sigma$ and $w_0$ are involutions on $\fraka$ they commute.
\end{proof}

\begin{lemma}\label{lem:SigmaOnN}
Assume that $G$ is simple and $P\neq G$. Then $\sigma|_\frakn=\id_\frakn$ if and only if $\sigma=\id_{\frakg}$.
\end{lemma}

\begin{proof}
The subalgebra generated by $\frakn$ and $\overline{\frakn}$ is a non-zero ideal in $\frakg$ since $\frakm$ and $\fraka$ leave $\frakn$ and $\overline{\frakn}$ invariant. Hence this subalgebra has to be $\frakg$ itself and it follows that $\frakn$ and $\overline{\frakn}$ generate $\frakg$. Now assume that $\sigma|_\frakn=\id_\frakn$. Since $\sigma$ commutes with $\theta$ and $\overline{\frakn}=\theta(\frakn)$ we also find that $\sigma|_{\overline{\frakn}}=\id_{\overline{\frakn}}$. But $\frakn$ and $\overline{\frakn}$ generate $\frakg$ and hence $\sigma$ is the identity on $\frakg$.
\end{proof}

\subsection{The double flag variety}

Consider the double flag variety $G/P\times H/P_H$. It carries a natural left-action by $H$ acting diagonally. For convenience write
\begin{equation*}
 \Delta(H) := \{(h,h):h\in H\}\subseteq G\times H.
\end{equation*}
We are interested in the structure of $\Delta(H)\backslash(G/P\times H/P_H)=\Delta(H)\backslash(G\times H)/(P\times P_H)$. In particular we are interested in cases where there exists an open (dense) orbit of $\Delta(H)$ on $G/P\times H/P_H$. We will see that this question is closely related to the orbit structure of $M_HA_H$ on $\frakn$. Note that
\begin{equation*}
 \frakn = \frakn^\sigma+\frakn^{-\sigma}
\end{equation*}
and this decomposition is stable under the adjoint action of $M_HA_H$.

\begin{lemma}\label{lem:TauDecompositionN}
The map $N_H\times\frakn^{-\sigma}\to N,\,(n,Y)\mapsto ne^Y$ is a diffeomorphism.
\end{lemma}

\begin{proof}
Since the nilpotent group $N$ is connected and simply-connected, the exponential map $\exp:\frakn\to N$ is a diffeomorphism. Hence we can define the inverse map
$$ \log=\exp^{-1}:N\to \frakn $$
and the square root
$$ N\to N,\,n\mapsto n^{\frac{1}{2}}:=e^{\frac{1}{2}\log n}, $$
both being smooth maps. For $n\in N$, put
\begin{align*}
n_H:=n(n^{-1}\sigma(n))^{\frac{1}{2}} \quad \text{and} \quad
Y:=\log \Bigl((n^{-1}\sigma(n))^{-\frac{1}{2}}\Bigr).
\end{align*}
Then
\begin{equation*}
 \sigma(n_H) = \sigma(n)(\sigma(n)^{-1}n)^{\frac{1}{2}} = n(n^{-1}\sigma(n))(n^{-1}\sigma(n))^{-\frac{1}{2}} = n(n^{-1}\sigma(n))^{\frac{1}{2}} = n_H
\end{equation*}
and hence $n_H\in N_H$. Similarly
\begin{equation*}
\sigma(e^Y)= \sigma((n^{-1}\sigma(n))^{-\frac{1}{2}}) = (\sigma(n)^{-1}n)^{-\frac{1}{2}} = (n^{-1}\sigma(n))^{\frac{1}{2}}=e^{-Y},
\end{equation*}
which implies $Y\in\frakn^{-\sigma}$.  
Therefore, $n\mapsto (n_H, Y)$ defines a smooth map $N\to N_H\times \frakn^{-\sigma}$.
It is easy to see that this map gives the inverse of $(n,Y)\mapsto ne^Y$.
\end{proof}

\begin{proposition}\label{prop:OrbitsDoubleFlag}
The map
\begin{align*}
 \Phi: \frakn^{-\sigma}/M_HA_H &\to \Delta(H)\backslash(G/P\times H/P_H),\\
 (M_HA_H\cdot X) &\mapsto \Delta(H)\cdot(e^X\tilde{w}_0P,P_H).
\end{align*}
is well-defined, injective and maps onto the $H$-orbits which are contained in the open dense subset $\Delta(H)\cdot(N\tilde{w}_0P,P_H)\subseteq G/P\times H/P_H$. It restricts to a bijection between the open $M_HA_H$-orbits in $\frakn^{-\sigma}$ and the open $H$-orbits in $G/P\times H/P_H$.
\end{proposition}

\begin{proof}
To see that the map $\Phi$ is well-defined let $X'=ma\cdot X$, where $X,X'\in\frakn^{-\sigma}$, $ma\in M_HA_H$. Since $\tilde{w}_0^{-1}M_HA_H\tilde{w}_0\subseteq MA\subseteq P$ and $M_HA_H\subseteq P_H$ we find
\begin{align*}
 \Delta(H)\cdot(e^{X'}\tilde{w}_0P,P_H) &= \Delta(H)\cdot(mae^X(ma)^{-1}\tilde{w}_0P,P_H)\\
 &= \Delta(H)\cdot(e^X\tilde{w}_0P,(ma)^{-1}P_H)\\
 &= \Delta(H)\cdot(e^X\tilde{w}_0P,P_H).
\end{align*}
To show the other claims note that
\begin{equation*}
 \Delta(H)\backslash(G\times H)/(P\times P_H)\cong P_H\backslash G/P
\end{equation*}
via the map induced by $G\times H\to G,\,(g,h)\mapsto h^{-1}g$. Hence the (open) orbits of $H$ on $G/P\times H/P_H$ correspond to the (open) orbits of $P_H$ on $G/P$. Via this isomorphism the map $\Phi$ takes the form
\begin{align*}
 \frakn^{-\sigma}/M_HA_H &\to P_H\backslash G/P,\\
 M_HA_H\cdot X &\mapsto P_H e^X\tilde{w}_0 P.
\end{align*}
Fix an orbit $\calO=M_HA_H\cdot X\subseteq\frakn^{-\sigma}$. Since $\tilde{w}_0^{-1}M_HA_H\tilde{w}_0\subseteq P$ we find
\begin{equation*}
 P_H e^X\tilde{w}_0 P = N_HM_HA_H e^X\tilde{w}_0 P = N_H e^{M_HA_H\cdot X}\tilde{w}_0 P = N_He^\calO \tilde{w}_0P
\end{equation*}
Now by Lemma~\ref{lem:TauDecompositionN} we have
\[ N=\bigsqcup_{\calO\in\frakn^{-\sigma}/M_HA_H}N_He^\calO. \]
Since the map $N\to G/P,\,n\mapsto n\tilde{w}_0P$ is injective this implies that $\Phi$ is injective and maps onto the $H$-orbits contained in $\Delta(H)\cdot(N\tilde{w}_0P,P_H)$. Further, since $P_H\subseteq P$ the open dense cell $N\tilde{w}_0P$ in the Bruhat decomposition \eqref{eq:BruhatDecomp} of $G/P$ is stable under $P_H$. Therefore, an open $P_H$-orbit in $G/P$ has to be contained in $N\cdot \tilde{w}_0P$ and is therefore in the image of $\Phi$. This completes the proof.
\end{proof}

\begin{remark}
The linearization technique that we applied in the proof of Proposition~\ref{prop:OrbitsDoubleFlag} was used before in the study of pairs $(G,H)$ of reductive groups with the open orbit property, i.e.\ $H$ has an open orbit on the double flag variety $G/P\times H/P_H$. When the parabolic subgroups are minimal, a classification for the group case $(G'\times G', \Delta(G'))$ was given in \cite{Kob95} (see also \cite{Dei06}) and for general symmetric pairs $(G,H)$ in \cite{KM}. For the complex case, \cite{HNOO12} gave a classification when $P$ or $P_H$ is a Borel subgroup.
\end{remark}

\subsection{Examples}\label{sec:ExamplesCompatibleSymmetricPairs}

We give some examples of symmetric pairs $(G,H)$ and parabolic subgroups $P\subseteq G$ which fulfil Conditions~\eqref{eq:CondG} and \eqref{eq:CondH} and we study the corresponding functions $a^\lambda$ and the $H$-orbits in the double flag variety $G/P\times H/P_H$.

\subsubsection{Rank one groups}\label{sec:ExRankOne}

Let $G=\SU(1,n;\FF)$ with $\FF=\RR,\CC,\HH$ and $n\geq2$ or $\FF=\OO$ and $n=2$. This means that $G=\SO(1,n)$, $G=\SU(1,n)$, $G=\Sp(1,n)$ or $G=\SU(1,2;\OO)=\upF_{4(-20)}$. These groups are all reductive of Harish-Chandra type. We choose the parabolic subgroup $P=P_{\min}=MAN\subseteq G$ such that $\fraka=\RR H_0$ with
\begin{equation*}
 H_0 = \left(\begin{array}{cc|c}0&1&\\1&0&\\\hline&&\0_{n-1}\end{array}\right)
\end{equation*}
and $\frakn=\frakg_\alpha\oplus\frakg_{2\alpha}$ for $\alpha\in\fraka^*$ with $\alpha(H_0)=1$. Then
\[ M = \begin{cases}\upS(\Delta\upU(1;\FF)\times\upU(n-1;\FF)) & \mbox{for $\FF=\RR,\CC,\HH$,}\\\Spin(7) & \mbox{for $\FF=\OO$,}\end{cases} \]
where
$$ \Delta\upU(1;\FF) = \left\{\left(\begin{array}{cc}z&\\&z\end{array}\right):z\in\upU(1;\FF)\right\}. $$
We identify $\overline{\frakn}\cong\FF^{n-1}\oplus\Im(\FF)$ by
\[ \FF^{n-1}\oplus\Im(\FF)\to\overline{\frakn},\,(x,z)\mapsto\left(\begin{array}{cc|c}z&z&x^*\\-z&-z&-x^*\\\hline x&x&\0_{n-1}\end{array}\right), \]
where $\Im(\FF)=\{z\in\FF:\overline{z}=-z\}$. (Note that $\Im(\FF)=0$ for $\FF=\RR$.) Under this identification the Lie bracket is given by
\[ [(x_1,z_1),(x_2,z_2)] = (0,x_1^*x_2-x_2^*x_1). \]
Hence $\frakn$ is abelian for $\FF=\RR$ and $2$-step nilpotent in the other cases. It is said to be of type $H$, a notion by Kaplan~\cite{Kap80} (see also \cite{CK84}). Since $\overline{N}$ is nilpotent we can identify it with its Lie algebra $\overline{\frakn}$. Under this identification the multiplication takes the form
\[ (x_1,z_1)\cdot(x_2,z_2) = (x_1+x_2,z_1+z_2+\tfrac{1}{2}(x_1^*x_2-x_2^*x_1)). \]
Abusing notation we denote by $(x,z)^{-1}=(-x,-z)$ the multiplicative inverse.

The group $M$ acts on $\frakn\cong\FF^{n-1}\oplus\Im\FF$ by the adjoint action as follows:
\begin{itemize}
\item For $\FF=\RR,\CC,\HH$ the factor $\upU(n-1;\FF)$ acts on $\FF^{n-1}$ by the defining representation (left multiplication) while the factor $\Delta\upU(1;\FF)\cong\upU(1;\FF)$ acts on $\FF^{n-1}$ by right multiplication where we identify $\upU(1;\FF)$ with the unit sphere in $\FF$. On $\Im\FF$ only the factor $\upU(1;\FF)$ acts, namely by conjugation.
\item For $\FF=\OO$ the group $M=\Spin(7)$ acts on $\FF\cong\RR^8$ by the spin representation and on $\Im\FF\cong\RR^7$ by the lift of the defining representation of $\SO(7)$.
\end{itemize}

Identifying $\fraka_\CC^*\cong\CC,\,\lambda\mapsto\lambda(H_0)$ we find that for $\lambda\in\fraka_\CC^*$
\begin{equation*}
 p^\lambda(X) = N(X)^{2\lambda}, \qquad X\in\overline{\frakn},
\end{equation*}
where
\begin{equation}
 N(x,z)=(|x|^4+4|z|^2)^{\frac{1}{4}}\label{eq:DefHTypeNorm}
\end{equation}
denotes the norm function of the $H$-type group $\overline{\frakn}$.

Now consider for $0<m<n$ the involution $\sigma$ given by conjugation with the matrix $\diag(\1_{m+1},-\1_{n-m})$ and put $H:=G^\sigma$. The possible symmetric pairs $(G,H)$ are given
\begin{align*}
 &(\SO(1,n),\upS(\upO(1,m)\times\upO(n-m))), &&(\Sp(1,n),\Sp(1,m)\times\Sp(n-m)),\\
 &(\SU(1,n),\upS(\upU(1,m)\times\upU(n-m))), &&(\upF_{4(-20)},\Spin(8,1)).
\end{align*}
Then the pair $(G,H)$ satisfies the above assumptions with
\begin{equation*}
 w_0 := \diag(1,-1,-1,1,\ldots,1).
\end{equation*}
We have $\fraka_H=\fraka=\RR H_0$ and
\begin{align*}
 M_H &= \begin{cases}\upS(\Delta\upU(1;\FF)\times\upU(m-1;\FF)\times\upU(n-m;\FF)) & \mbox{for $\FF=\RR,\CC,\HH$,}\\\Spin(7) & \mbox{for $\FF=\OO$,}\end{cases}\\
 \frakn_H &= \FF^{m-1}\oplus\Im(\FF).
\end{align*}

To determine the decomposition of $G/P\times H/P_H$ into $H$-orbits note that $\frakn^{-\sigma}$ decomposes into two $M_HA_H$-orbits, the origin and its complement which is open dense. In fact, for $\FF=\RR,\CC,\HH$ we have $\frakn^{-\sigma}=\FF^{n-m}$ and $\upU(n-m;\FF)$ acts transitively on the unit sphere in $\FF^{n-m}$. For $\FF=\OO$ we have $\frakn^{-\sigma}=\frakg_\alpha=\OO\cong\RR^8$ on which $M_H=M=\Spin(7)$ acts by the spin representation. Under this action $M_H$ acts transitively on the unit sphere in $\RR^8$ and hence $M_HA_H$ acts on $\frakn^{-\sigma}$ with an open dense orbit. (This also follows from Kostant's $2$-transitivity Theorem~\cite[Theorem 3]{Kos69}.) With Proposition~\ref{prop:OrbitsDoubleFlag} it is easy to see that
\[ G\times H = \calO_1\sqcup\calO_2\sqcup\calO_3 \]
with
\begin{align*}
 \calO_1 &= \Delta(H)(\1,\1)(P\times P_H),\\
 \calO_2 &= \Delta(H)(\tilde{w}_0,\1)(P\times P_H),\\
 \calO_3 &= \Delta(H)(n_0\tilde{w}_0,\1)(P\times P_H),
\end{align*}
where $n_0\in N\setminus N_H$ arbitrary. These orbits define a stratification in the sense that
\[ \overline{\calO_j} = \calO_1\sqcup\ldots\sqcup\calO_j. \]

One could also consider different symmetric pairs, e.g.\ the pair $(G,H)=(\SU(1,n),\SO_0(1,n))$. It also satisfies assumptions \eqref{eq:CondG} and \eqref{eq:CondH}, but in this case $M_HA_H=\RR_+\SO(n-1)$ does not have an open orbit on $\frakn^{-\sigma}=i\RR^{n-1}\oplus i\RR$.

\subsubsection{Product situation}\label{sec:ExProduct}

Let $G'$ be a real reductive Lie group in the Harish-Chandra class. Put $G=G'\times G'$ and $\sigma(g_1,g_2)=(g_2,g_1)$. Then $H=G^\sigma=\Delta(G')$, the diagonal in $G$. Let $P'=M'A'N'\subseteq G'$ be any parabolic subgroup of $G'$ which is conjugate to its opposite parabolic subgroup via the longest element $w_0'$ in the Weyl group for $G'$. Then $P=P'\times P'$ is a parabolic subgroup of $G$ satisfying the assumptions \eqref{eq:CondG} and \eqref{eq:CondH} with $w_0=(w_0',w_0')$. Denote by $a'(\blank)$ the $a$-function of $G'$ with respect to $P'$. Then for $\lambda=(\lambda_1,\lambda_2)\in\fraka_\CC^*=(\fraka'_\CC)^*\oplus(\fraka'_\CC)^*$ we have
\begin{equation*}
 a(\tilde{w}_0^{-1}(g_1,g_2))^\lambda = a'(\tilde{w}_0'^{-1}g_1)^{\lambda_1}a'(\tilde{w}_0'^{-1}g_2)^{\lambda_2}, \qquad (g_1,g_2)\in G.
\end{equation*}

Further we have
\[ \frakn^{-\sigma} = \{(X,-X):X\in\frakn'\} \cong \frakn' \]
as $M'A'$-representations. Hence, there are open $G'$-orbits in the triple flag variety $G/P\times H/P_H\cong G'/P'\times G'/P'\times G'/P'$ if and only if $M'A'$ has open orbits on $\frakn'$. For $P'$ a minimal parabolic subgroup this is only the case for $\frakg'$ a direct sum of copies of $\so(1,n)$ by \cite{Kob95} and \cite[Theorem 3.1]{Dei06}. In general this question is more involved.

\subsubsection{Maximal parabolic subgroups with abelian nilradical}\label{sec:SymmRSpaces}

Let $G$ be a real reductive Lie group in the Harish-Chandra class which possesses a maximal parabolic subgroup $P$ with abelian nilradical such that $P$ and $\overline{P}$ are conjugate (see Table~\ref{tb:MaxParabolics} for a classification of the corresponding Lie algebras modulo center). In many cases one can find symmetric pairs $(G, H)$ such that $H$ is the product of two versions of $G$ of lower rank and fulfils Condition~\eqref{eq:CondH}.

\begin{table}[h]
\begin{center}
\begin{tabular}{|c|c|c|}
  \cline{1-3}
  $\frakg$ & $\frakm$ & $\frakn$\\
  \hline\hline
  $\sp(n,\RR)$ & $\sl(n,\RR)$ & $\Sym(k,\RR)$\\
  $\su(n,n)$ & $\sl(n,\CC)$ & $\Herm(n,\CC)$\\
  $\so^*(4n)$ & $\su^*(2n)$ & $\Herm(n,\HH)$\\
  $\so(2,n)$ & $\so(1,n-1)$ & $\RR^{1,n-1}$\\
  $\mathfrak{e}_{7(-25)}$ & $\mathfrak{e}_{6(-26)}$ & $\Herm(3,\OO)$\\
  \hline
  $\sl(2n,\RR)$ & $\sl(n,\RR)\oplus\sl(n,\RR)$ & $M(n,\RR)$\\
  $\so(2n,2n)$ & $\sl(2n,\RR)$ & $\Skew(2n,\RR)$\\
  $\so(p,q)$ & $\so(p-1,q-1)$ & $\RR^{p-1,q-1}$\\
  $\mathfrak{e}_{7(7)}$ & $\mathfrak{e}_{6(6)}$ & $\Herm(3,\OO_s)$\\
  \hline
  $\sp(n,\CC)$ & $\sl(n,\CC)\oplus i\RR$ & $\Sym(n,\CC)$\\
  $\sl(2n,\CC)$ & $\sl(n,\CC)\oplus\sl(n,\CC)\oplus i\RR$ & $M(n,\CC)$\\
  $\so(4n,\CC)$ & $\sl(2n,\CC)\oplus i\RR$ & $\Skew(2n,\CC)$\\
  $\so(n+2,\CC)$ & $\so(n,\CC)\oplus i\RR$ & $\CC^n$\\
  $\mathfrak{e}_7(\CC)$ & $\mathfrak{e}_6(\CC)\oplus i\RR$ & $\Herm(3,\OO)_\CC$\\
  \hline
  $\sp(n,n)$ & $\su^*(2n)$ & $\Sym(2n,\CC)\cap M(n,\HH)$\\
  $\su^*(4n)$ & $\su^*(2n)\oplus\su^*(2n)$ & $M(n,\HH)$\\
  $\so(n+1,1)$ & $\so(n)$ & $\RR^{n,0}$\\
  \hline
\end{tabular}
\caption{Maximal parabolic subalgebras of semisimple Lie algebras with abelian nilradical\label{tb:MaxParabolics}}
\end{center}
\end{table}

\begin{example}\label{ex:JAExamples}
\begin{enumerate}
\item Let $G=\Sp(n,\RR)$ and $P=\GL(n,\RR)\ltimes\Sym(n,\RR)$ the Siegel parabolic subgroup. Then for $\lambda\in\fraka_\CC^*\cong\CC$ we have
\begin{align}\label{eq:pLambdaFormula} p^\lambda(X) = |\Det(X)|^\lambda, \qquad X\in\overline{\frakn}=\Sym(n,\RR). \end{align}
Let $H=\Sp(m,\RR)\times\Sp(n-m,\RR)$ then $P_H$ is the product of the corresponding Siegel parabolic subgroups. Hence $\frakn^{-\sigma}=M(m\times(n-m),\RR)$ on which $M_HA_H=\GL(m,\RR)\times\GL(n-m,\RR)$ acts by multiplication from the left and right. This action has an open dense orbit whence there is an open dense $H$-orbit in $G/P\times H/P_H$.
\item Let $G=\SL(2n,\FF)$, $\FF=\RR,\CC,\HH$, where $\SL(2n,\HH)=\SU^*(4n)$. The parabolic subgroup $P=\upS(\GL(n,\FF)\times\GL(n,\FF))\ltimes M(n,\FF)$ is maximal with abelian nilradical. Then for $\lambda\in\fraka_\CC^*\cong\CC$ we have
\[ p^\lambda(X) = |\Det(X)|^\lambda, \qquad X\in\overline{\frakn}=M(n,\FF), \]
where $\Det(X)$ denotes the complex determinant in the case $\overline{\frakn}=M(n,\HH)\subseteq M(2n,\CC)$. Let $H=\upS(\GL(2m,\FF)\times\GL(2n-2m,\FF))$ embedded such that $P_H=P\cap H$ is the product of two versions of $P$ of rank $m$ and $n-m$ and a central factor. Here $\frakn^{-\sigma}=M(m\times(n-m),\FF)\oplus M((n-m)\times m,\FF)$ and $M_HA_H=\upS(\GL(m,\FF)\times\GL(n-m,\FF)\times\GL(m,\FF)\times\GL(n-m,\FF))$. The first and the last factor of $M_HA_H$ act by left and right multiplication on the first summand of $\frakn^{-\sigma}$ and the second and third factor accordingly on the second summand. Hence $M_HA_H$ has an open dense orbit on $\frakn^{-\sigma}$ and therefore $H$ has an open dense orbit on $G/P\times H/P_H$.
\item Let $G=\GL(4n,\RR)$ and $P=(\GL(2n,\RR)\times\GL(2n,\RR))\ltimes M(2n,\RR)$. Then one can embed $H=\GL(2n,\CC)$ into $G$ such that $P_H=(\GL(n,\CC)\times\GL(n,\CC))\ltimes M(n,\CC)$. Then $\frakn^{-\sigma}\simeq M(n,\CC)$ on which $M_HA_H=\GL(n,\CC)\times\GL(n,\CC)$ acts by left and right multiplication, having the invertible matrices in $M(n,\CC)$ as the unique open orbit. Hence $H$ has an open dense orbit on $G/P\times H/P_H$.
\item An example of a slightly different nature is given by $G=\SO(p,q)$ and $P$ the maximal parabolic subgroup with $\frakm=\so(p-1,q-1)$ and $\frakn=\RR^{p+q-2}$. Let $Q$ be the quadratic form on $\frakn$ given by
\[ Q(X) = X_1^2+\cdots+X_{p-1}^2-X_p^2-\cdots-X_{p+q-2}^2. \]
For $\lambda\in\fraka_\CC^*\cong\CC$ we have
\[ p^\lambda(X) = |Q(X)|^\lambda, \qquad X\in\overline{\frakn}=\RR^{p+q-2}. \]
Embed $H=\upS(\upO(p',q')\times\upO(p-p',q-q'))$ into $G$ such that $P_H=P\cap H$ is an open subgroup of the product of $\upO(p-p',q-q')$ with the corresponding maximal parabolic subgroup of $\upO(p',q')$, i.e.\ $\frakm_H=\so(p'-1,q'-1)\oplus\so(p-p',q-q')$. In this case $\frakn^{-\sigma}=\RR^{(p-p')+(q-q')}$ on which $\RR_+\upO(p-p',q-q')\subseteq M_HA_H$ acts with an open dense orbit. Hence $H$ has an open dense orbit on $G/P\times H/P_H$.
\end{enumerate}
\end{example}
\section{The invariant kernel}\label{sec:InvriantKernel}

We study a family of $H$-invariant singular integral kernels on $G\times H$.

\subsection{Definition of the integral kernel}

For $\alpha,\beta\in\fraka_\CC^*$ we introduce the following kernel:
\begin{equation}
 K_{\alpha,\beta}(g,h) := a(\tilde{w}_0^{-1}g^{-1}h)^\alpha a(\tilde{w}_0^{-1}g^{-1}\sigma(g))^\beta\label{eq:DefKernel}
\end{equation}
for $g\in G$, $h\in H$, whenever the expression on the right hand side is defined.

\begin{example}\label{ex:KernelExamples}
\begin{enumerate}
\item For $\sigma=\id_G$ the kernel is defined only if $\beta=0$. In this case
\[ K_{\alpha,\beta}(g,h) = a(\tilde{w}_0^{-1}g^{-1}h)^\alpha, \qquad g,h\in G, \]
the kernel of the classical Knapp--Stein intertwiners, see \eqref{eq:ClassicalKSasConvolution}.
\item For $G=G'\times G'$, $\sigma(g_1,g_2)=(g_2,g_1)$, $P=P'\times P'$ as in Section~\ref{sec:ExProduct} we have by Remark~\ref{rem:PropertiesAFct1}~(2) with $\alpha=(\alpha_1,\alpha_2),\beta=(\beta_1,\beta_2)\in\fraka_\CC^*=(\fraka')_\CC^*\oplus(\fraka')_\CC^*$:
\begin{multline*}
 \hspace{1.5cm}K_{\alpha,\beta}((g_1,g_2),g_3) = a(\tilde{w}_0'^{-1}g_1^{-1}g_2)^{\beta_1-w_0'\beta_2} a(\tilde{w}_0'^{-1}g_2^{-1}g_3)^{\alpha_2} \\
 a(\tilde{w}_0'^{-1}g_3^{-1}g_1)^{-w_0'\alpha_1}.
\end{multline*}
These are the triple kernels considered in \cite{BC12,BSKZ,BR04,CKOP11,CO11,Dei06,Oks73}.
\item For $G=\SU(1,n;\FF)$, $H=\upS(\upU(1,m;\FF)\times\upU(n-m;\FF))$ and $P$ as in Section~\ref{sec:ExRankOne} the kernel $K_{\alpha,\beta}(g,h)$ is in the flat picture given by
\[ K_{\alpha,\beta}(e^X,e^Y) = N(X^{-1}\cdot Y)^{2\alpha} N(X^{-1}\cdot \sigma(X))^{2\beta}, \]
where $X\in\overline{\frakn}=\FF^{n-1}\oplus\Im\FF$ and $Y\in\overline{\frakn}_H=\FF^{m-1}\oplus\Im\FF$ and $N(X)$ denotes the norm function on $\overline{\frakn}$ as defined in \eqref{eq:DefHTypeNorm}. In the special case $\FF=\RR$ we obtain the kernel
\[ K_{\alpha,\beta}(e^X,e^Y) = (|X'-Y|^2+|X''|^2)^\alpha |2X''|^{2\beta}, \]
where $X=(X',X'')\in\RR^{m-1}\times\RR^{n-m}=\RR^{n-1}$ and $Y\in\RR^{m-1}$. For $m=n-1$ this kernel was studied in detail by Kobayashi--Speh~\cite{KS,KS14} (see also \cite{Kob13,MO12}).
\end{enumerate}
\end{example}

\begin{remark}\label{rem:KernelDependenceOnBeta}
By Lemma~\ref{lem:SigmaW0Commute} and \eqref{eq:AlambdaFctInverse} we have
\begin{align*}
 a(\tilde{w}_0^{-1}g^{-1}\sigma(g))^\beta &= a(\tilde{w}_0^{-1}\sigma(g)^{-1}g)^{-w_0\beta} = a(\sigma(\tilde{w}_0^{-1})g^{-1}\sigma(g))^{-w_0\sigma\beta}\\
 &= a(\tilde{w}_0^{-1}g^{-1}\sigma(g))^{-w_0\sigma\beta}.
\end{align*}
Hence $a(\tilde{w}_0^{-1}g^{-1}\sigma(g))^{\beta+w_0\sigma\beta}=1$ and therefore the kernel $K_{\alpha,\beta}(g,h)$ does not depend on the values of $\beta$ on $\fraka^{w_0\sigma}=\{H\in\fraka:w_0\sigma H=H\}$.
\end{remark}

\subsection{Properties of the kernel}

The kernel $K_{\alpha,\beta}(g,h)$ has the following equivariance properties:

\begin{proposition}\label{prop:KernelEquivariance}
\begin{enumerate}[(1)]
\item The kernel $K_{\alpha,\beta}(g,h)$ is left-invariant under $\Delta(H)$, i.e.
\begin{equation*}
 K_{\alpha,\beta}(h'g,h'h) = K_{\alpha,\beta}(g,h) \qquad \mbox{for }g\in G,\,h,h'\in H.
\end{equation*}
\item The kernel $K_{\alpha,\beta}(g,h)$ satisfies the following equivariance property:
\begin{equation}
 K_{\alpha,\beta}(gman,hm_Ha_Hn_H) = a^{-w_0\alpha+\sigma\beta-w_0\beta}a_H^\alpha K_{\alpha,\beta}(g,h)
\end{equation}
for $g\in G$, $h\in H$ and $man\in MAN$, $m_Ha_Hn_H\in M_HA_HN_H$.
\end{enumerate}
\end{proposition}

\begin{proof}
\begin{enumerate}[(1)]
\item This is clear from the definition since $\sigma(h')=h'$ for $h'\in H$.
\item Direct computation using \eqref{eq:EquivarianceAFunction} yields
\begin{align*}
 & K_{\alpha,\beta}(gman,hm_Ha_Hn_H)\\
 ={}& a(\tilde{w}_0^{-1}n^{-1}a^{-1}m^{-1}g^{-1}hm_Ha_Hn_H)^\alpha a(\tilde{w}_0^{-1}n^{-1}a^{-1}m^{-1}g^{-1}\sigma(g)\sigma(m)\sigma(a)\sigma(n))^\beta\\
 ={}& a(\overline{n}(^{w_0}a^{-1})m'w_0^{-1}g^{-1}hm_Ha_Hn_H)^\alpha a(\overline{n}(^{w_0}a^{-1})m'\tilde{w}_0^{-1}g^{-1}\sigma(g)\sigma(m)\sigma(a)\sigma(n))^\beta\\
 ={}& a^{-w_0\alpha} a(\tilde{w}_0^{-1}g^{-1}h)^\alpha a_H^\alpha a^{-w_0\beta} a(\tilde{w}_0^{-1}g^{-1}\sigma(g))^\beta \sigma(a)^\beta\\
 ={}& a^{-w_0\alpha+\sigma\beta-w_0\beta}a_H^\alpha K_{\alpha,\beta}(g,h),
\end{align*}
where $\overline{n}=\tilde{w}_0^{-1}n^{-1}\tilde{w}_0\in\overline{N}$ and $m'=\tilde{w}_0^{-1}m^{-1}\tilde{w}_0\in M$.\qedhere
\end{enumerate}
\end{proof}

\subsection{Domain of definition}\label{sec:KernelDomain}

The kernel $K_{\alpha,\beta}(g,h)$ is defined at $(g,h)\in G\times H$ whenever
\begin{equation}
 g^{-1}h,g^{-1}\sigma(g)\in \tilde{w}_0\overline{N}MAN=N\tilde{w}_0MAN.\label{eq:KernelDefCondition}
\end{equation}
Note that $N\tilde{w}_0MAN=P\tilde{w}_0P$ is the open dense cell in the Bruhat decomposition~\eqref{eq:BruhatDecomp}. The condition \eqref{eq:KernelDefCondition} on $(g,h)\in G\times H$ is right-invariant under $P\times P_H$ and the following set is well-defined:
\begin{equation*}
 \calD := \{(gP,hP_H)\in G/P\times H/P_H:g^{-1}h,g^{-1}\sigma(g)\in N\tilde{w}_0MAN\}.
\end{equation*}
We have that $K_{\alpha,\beta}(g,h)$ is defined at $(g,h)\in G\times H$ if and only if $(gP,hP_H)\in\calD$.

\begin{lemma}\label{lem:PropertiesD}
\begin{enumerate}[(1)]
\item $\calD$ is an open subset of $G/P\times H/P_H$.
\item $\calD$ is left-invariant under $\Delta(H)$.
\end{enumerate}
\end{lemma}

\begin{proof}
\begin{enumerate}[(1)]
\item Since $N\tilde{w}_0MAN\subseteq G$ is open and the maps $G\times H\to G,\,(g,h)\mapsto g^{-1}h$ and $G\to G,\,g\mapsto g^{-1}\sigma(g)$ are continuous, the inverse image of $\calD$ under the product of the projections $G\to G/P$ and $H\to H/P_H$ is open in $G\times H$. Hence $\calD$ is open in $G/P\times H/P_H$.
\item This is clear as $\sigma(h)=h$ for $h\in H$.\qedhere
\end{enumerate}
\end{proof}

\begin{proposition}\label{prop:DDense}
The following conditions are equivalent:
\begin{enumerate}[(1)]
\item $\calD\subseteq G/P\times H/P_H$ is open dense,
\item $\calD\neq\emptyset$,
\item $\exp(\overline{\frakn}^{-\sigma})\cap N\tilde{w}_0MAN\neq\emptyset$,
\item $p^\lambda|_{\overline{\frakn}^{-\sigma}}\neq0$ for some/all $\lambda\in\fraka_{+,\reg}^*\cap\Lambda^+(\fraka)$.
\end{enumerate}
\end{proposition}

\begin{proof}
We prove (1)$\Rightarrow$(2)$\Rightarrow$(3)$\Rightarrow$(4)$\Rightarrow$(1). The direction (1)$\Rightarrow$(2) is trivial.\\
Next assume (2), $\calD\neq\emptyset$. Since $\calD$ is open it intersects the open dense set $(\overline{N}\cdot \1 P)\times H/P_H$ non-trivially. In particular there exists $\overline{n}\in\overline{\frakn}$ such that $\overline{n}^{-1}\sigma(\overline{n})\in N\tilde{w}_0MAN$. By Lemma~\ref{lem:TauDecompositionN} we can write $\overline{n}=\overline{n}_He^X$ with $\overline{n}_H\in\overline{N}_H$ and $X\in\overline{\frakn}^{-\sigma}$. Then $\overline{n}^{-1}\sigma(\overline{n})=e^{-2X}$ and hence $e^{-2X}\in N\tilde{w}_0MAN$ which shows (3).\\
Now assume (3), we have $e^X\in N\tilde{w}_0MAN$ for some $X\in\overline{\frakn}^{-\sigma}$ and let $\lambda\in\fraka_{+,\reg}^*\cap\Lambda^+(\fraka)$. By Lemma~\ref{lem:MatrixCoeffPolynomialOnNbar}~(3) we have $p^\lambda(Y)\neq0$ if and only if $e^Y\in N\tilde{w}_0MAN$. Hence $p^\lambda(X)\neq0$ which shows (4).\\
Finally assume (4), the restriction of $p^\lambda$ to $\overline{\frakn}^{-\sigma}$ is non-zero for some $\lambda\in\fraka_{+,\reg}^*\cap\Lambda^+(\fraka)$. Note that such $\lambda$ exist by Proposition~\ref{lem:BoundednessAFct}~(4). Then $(p^\lambda|_{\overline{\frakn}^{-\sigma}})^{-1}(\RR\setminus\{0\})\subseteq\overline{\frakn}^{-\sigma}$ is open dense and consequently, by Lemma~\ref{lem:TauDecompositionN},
\begin{align*}
 U :={}& \{\overline{n}\in\overline{N}:\overline{n}^{-1}\sigma(\overline{n})\in N\tilde{w}_0MAN\}\\
 ={}& \{\overline{n}_He^X:\overline{n}_H\in\overline{N}_H, X\in\overline{\frakn}^{-\sigma}, p^\lambda(-2X)\neq0\}
\end{align*}
is open dense in $\overline{N}$. Consider the topological isomorphism
\begin{equation*}
 \overline{N}\times\overline{N}_H\to\overline{N}\times\overline{N}_H,\,(\overline{n},\overline{n}_H)\mapsto(\overline{n}^{-1}\overline{n}_H,\overline{n}_H)
\end{equation*}
and denote by $V\subseteq\overline{N}\times\overline{N}_H$ the preimage of the open dense set $(\overline{N}\cap N\tilde{w}_0MAN)\times\overline{N}_H$. Then $U\times\overline{N}_H$ and $V$ are both open dense in $\overline{N}\times\overline{N}_H$ and hence their intersection is open dense. The image of the intersection under the canonical projection $G\times H\to G/P\times H/P_H$ is therefore open dense in $G/P\times H/P_H$, but it is also contained in $\calD$ which shows that $\calD$ is dense in $G/P\times H/P_H$. This proves (1) and hence the equivalence of all three statements follows.
\end{proof}

\begin{remark}
The kernel $K_{\alpha,\beta}(g,h)$ is still defined for $\calD=\emptyset$ if one imposes the condition $\beta=0$. In this case it is simply the kernel of the classical Knapp--Stein intertwiner restricted to $G\times H\subseteq G\times G$, see \eqref{eq:ClassicalKSasConvolution}.
\end{remark}

\begin{corollary}\label{cor:DDense}
Assume that either
\begin{enumerate}[(1)]
\item $G$ is a simple rank one group, $P=P_{\min}$ and $G_0\nsubseteq H$ or
\item $G=G'\times G'$, $H=\Delta(G')$ the diagonal and $P=P'\times P'$.
\end{enumerate}
Then $\calD\subseteq G/P\times H/P_H$ is open dense.
\end{corollary}

\begin{proof}
We use criterion (3) in Proposition~\ref{prop:DDense}.
\begin{enumerate}[(1)]
\item Since $G$ is of rank one we have $W=\{1,w_0\}$ and hence there are only two Bruhat cells $N\tilde{w}_0MAN$ and $MAN$. Further, $\overline{N}\cap MAN=\{\1\}$. Therefore $e^X\in N\tilde{w}_0MAN$, $X\in\overline{\frakn}$, if and only if $X\neq0$. Since $G_0\nsubseteq H$ we have $\sigma|_{\overline{\frakn}}\neq\id_{\overline{\frakn}}$ by Lemma~\ref{lem:SigmaOnN} and hence $\overline{\frakn}^{-\sigma}\neq0$. Therefore $\exp(\overline{\frakn}^{-\sigma})\cap N\tilde{w}_0MAN\neq\emptyset$.
\item We have $\overline{\frakn}^{-\sigma}=\{(X,-X):X\in\overline{\frakn}'\}$ and hence
\begin{equation*}
 \exp(\overline{\frakn}^{-\sigma}) = \{(\overline{n},\overline{n}^{-1}):\overline{n}\in\overline{N}'\}.
\end{equation*}
Further
\begin{equation*}
 N\tilde{w}_0MAN = N'\tilde{w}'_0M'A'N'\times N'\tilde{w}'_0M'A'N'.
\end{equation*}
Since $U:=\overline{N}'\cap N'\tilde{w}'_0M'A'N'$ is open dense in $\overline{N}'$ the intersection $U\cap U^{-1}$ is non-empty and for every $\overline{n}\in U\cap U^{-1}$ we have $(\overline{n},\overline{n}^{-1})\in\exp(\overline{\frakn}^{-\sigma})\cap N\tilde{w}_0MAN$.\qedhere
\end{enumerate}
\end{proof}

\begin{example}\label{ex:DDenseForSp(n,R)}
Section~\ref{sec:SymmRSpaces} provides a big class of examples where we can easily check whether the subset $\calD$ is dense in $G/P\times H/P_H$. We illustrate this in the case $G=\Sp(n,\RR)$ with $P=\GL(n,\RR)\ltimes\Sym(n,\RR)$ the Siegel parabolic subgroup and $H=\Sp(m,\RR)\times\Sp(n-m,\RR)$, see Example~\ref{ex:JAExamples}~(1). Here $\overline{\frakn}=\Sym(n,\RR)$ and then  Lemma~\ref{lem:MatrixCoeffPolynomialOnNbar}~(3) and \eqref{eq:pLambdaFormula} imply that for $X\in\overline{\frakn}$ we have $\exp(X)\in N\tilde{w}_0MAN$ if and only if $X$ is an invertible matrix. Now
\begin{equation*}
 \overline{\frakn}^{-\sigma} = \left\{\left(\begin{array}{cc}0&Y\\Y^T&0\end{array}\right):Y\in M(m\times(n-m),\RR)\right\}.
\end{equation*}
Therefore $\overline{\frakn}^{-\sigma}$ contains invertible matrices if and only if $n=2m$. Thus $\calD\neq\emptyset$ if and only if $n=2m$. The other cases in Example~\ref{ex:JAExamples} can be treated similarly.
\end{example}
\section{Intertwining operators between principal series}

We study intertwining operators between spherical principal series representations of $G$ and $H$.

\subsection{Induced representations}\label{sec:InducedRepresentations}

For $\nu\in\fraka_\CC^*$ we define the induced representation (normalized smooth parabolic induction)
\[ I^G(\nu) := \Ind_P^G(\1\otimes e^\nu\otimes\1). \]
Here $G$ acts by left-translations on the representation space
\[ I^G(\nu) = \{f\in C^\infty(G):f(gman)=a^{-\nu-\rho}f(g)\,\forall\,g\in G,man\in MAN\} \]
which is endowed with the topology induced from $C^\infty(G)$. Note that a function $f\in I^G(\nu)$ is uniquely determined by its values on $K$ and the restriction map defines a topological isomorphism
\[ I^G(\nu)\to C^\infty(X) \]
with $X=K/(M\cap K)$. The corresponding representation $\pi_\nu$ of $G$ on $C^\infty(X)$ is called the \textit{compact picture} and is explicitly given by
\begin{equation}
 \pi_\nu(g)f(k) = e^{-(\nu+\rho)H(g^{-1}k)}f(\kappa(g^{-1}k)), \qquad g\in G,\,k\in K.\label{eq:ActionOnCinftyK/M}
\end{equation}

Similarly, for $\nu'\in(\fraka_H)_\CC^*$ we also consider the induced representation
\begin{align*}
 I^H(\nu') &:= \Ind_{P_H}^H(\1\otimes e^{\nu'}\otimes\1)
\end{align*}
and its realization $\tau_{\nu'}$ on $C^\infty(X_H)$ with $X_H=K_H/(M_H\cap K_H)$.

\subsection{Intertwining integrals}

We use the kernels $K_{\alpha,\beta}(g,h)$ to construct intertwining operators $\pi_\nu|_H\to\tau_{\nu'}$. For this we have to assume that the domain of definition $\calD$ of $K_{\alpha,\beta}(g,h)$ is an open dense subset of $G\times H$. In view of Proposition~\ref{prop:DDense} we make the following general assumption:
\begin{equation}
 \calD\neq\emptyset, \label{eq:CondD}\tag{D}
\end{equation}
assuring that $\calD$ is open dense. In the spirit of the classical Knapp--Stein operators \eqref{eq:ClassicalKSasConvolution} we would like to put for $\alpha,\beta\in\fraka_\CC^*$ and $f\in C^\infty(X)$:
\begin{equation}
 A(\alpha,\beta)f(k_H) := \int_K K_{\alpha,\beta}(k,k_H)f(k) \td k, \qquad k_H\in K_H.\label{eq:DefIntertwinerCptPicture}
\end{equation}
Since the integral kernel $K_{\alpha,\beta}(g,h)$ is in general singular this integral does not converge for all parameters $\alpha,\beta$. Further, from this expression it is a priori not clear whether $A(\alpha,\beta)f$ defines a smooth function on $K_H$, even if we assume convergence of the integral. We rewrite \eqref{eq:DefIntertwinerCptPicture} using the $H$-invariance of $K_{\alpha,\beta}(g,h)$:
\begin{equation}
 A(\alpha,\beta)f(k_H) = \int_K K_{\alpha,\beta}(k_H^{-1}k,\1)f(k) \td k = \int_K K_{\alpha,\beta}(k,\1)f(k_Hk) \td k.\label{eq:DefIntertwinerRewritten}
\end{equation}
This expression suggests the investigation of the function
$$ \tilde{K}_{\alpha,\beta}(k):=K_{\alpha,\beta}(k,\1). $$
Note that $\tilde{K}_{\alpha,\beta}$ corresponds to $K_{\alpha,\beta}$ via the isomorphism $\Delta(H)\backslash(G\times H)\cong G,\,(g,h)\mapsto h^{-1}g$.

Recall the cone $\fraka_+^*\subseteq\fraka^*$ defined in \eqref{eq:Defa^*_+} and its interior $\fraka_{+,\reg}^*$. For a smooth manifold $Y$ we denote by $\calD'(Y):=C_c^\infty(Y)'$ the space of distributions on $Y$ endowed with the weak-$\star$ topology.

\begin{theorem}\label{thm:MeromorphicPropertiesKtilde}
\begin{enumerate}
\item For $\alpha,\beta\in\fraka_\CC^*$ with $\Re\alpha,\Re\beta\in\fraka_+^*$ the function $\tilde{K}_{\alpha,\beta}$ is locally integrable on $K$ and hence defines a non-zero distribution in $\calD'(K)$. The map
$$ (\fraka_+^*+i\fraka^*)\oplus(\fraka_+^*+i\fraka^*)\to\calD'(K), \quad (\alpha,\beta)\mapsto\tilde{K}_{\alpha,\beta} $$
is holomorphic on $(\fraka_{+,\reg}^*+i\fraka^*)\oplus(\fraka_{+,\reg}^*+i\fraka^*)$.
\item The distribution $\tilde{K}_{\alpha,\beta}$ extends meromorphically in the parameters $\alpha,\beta\in\fraka_\CC^*$. More precisely, there exist $X_1,\ldots,X_M$, $Y_1,\ldots,Y_M\in\fraka$ and $d_1,\ldots,d_M\in\ZZ$ such that the map
\begin{equation*}
 (\alpha,\beta)\mapsto\prod_{j=1}^M\Gamma(\alpha(X_j)+\beta(Y_j)+d_j)^{-1} \cdot \tilde{K}_{\alpha,\beta}
\end{equation*}
extends to an entire function $\fraka_\CC^*\times\fraka_\CC^*\to\calD'(K)$.
\end{enumerate}
\end{theorem}

For the proof we use a general result on the meromorphic continuation of complex power functions (see \cite{Ati70} and \cite{BG69} for a proof using Hironaka's resolution of singularities, see also \cite[Theorem 1]{KK79} and \cite[Th\'{e}or\`{e}me 2.1]{Sab87}):

\begin{theorem}\label{thm:KashiwaraKawai}
Let $Y$ be a compact real analytic manifold with a volume form $\td y$ and let $u_1,\ldots,u_N$ be non-negative real-valued real analytic functions on $Y$. Then the distribution $u_1^{s_1}\cdots u_N^{s_N}\in\calD'(Y)$ defined for $s_1,\ldots,s_N\in\CC$ with $\Re s_k\geq0$ by
\begin{equation*}
 \langle u_1^{s_1}\cdots u_N^{s_N},\varphi\rangle = \int_Y \varphi(y)u_1(y)^{s_1}\cdots u_N(y)^{s_N} \td y, \qquad \varphi\in C_c^\infty(Y),
\end{equation*}
extends meromorphically in the parameters $s_1,\ldots,s_N\in\CC$. More precisely, there exist $\alpha_{jk}\in\NN_0$ and $\beta_j\in\ZZ$, $j=1,\ldots,M$, $k=1,\ldots,N$, such that the map
\begin{equation*}
 (s_1,\ldots,s_N)\mapsto\prod_{j=1}^M\Gamma\left(\sum_{k=1}^N\alpha_{jk}s_k+\beta_j\right)^{-1}u_1^{s_1}\cdots u_N^{s_N}
\end{equation*}
extends to an entire function $\CC^N\to\calD'(Y)$.
\end{theorem}

In fact, the result in this formulation can be derived from \cite[Th\'{e}or\`{e}me  2.1]{Sab87} by choosing a finite number of coordinate patches and a corresponding partition of unity.

\begin{proof}[{Proof of Theorem~\ref{thm:MeromorphicPropertiesKtilde}}]
\begin{enumerate}
\item By the assumption~\eqref{eq:CondD} and Proposition~\ref{prop:DDense} the function $\tilde{K}_{\alpha,\beta}$ is defined on an open dense subset of $K$ and has strictly positive values on it. Since $\Re\alpha,\Re\beta\in\fraka_+^*$ this function is bounded by Lemma~\ref{lem:BoundednessAFct}~(1) and hence defines a distribution on $K$. Holomorphic dependence on $\alpha$ and $\beta$ follows from part (2).
\item We apply Theorem~\ref{thm:KashiwaraKawai} to $Y=K$ with the Haar measure $\td y=\td k$. Note that by Lemma~\ref{lem:BoundednessAFct}~(3) there exists a basis of $\fraka_\CC^*$ consisting of elements $\varpi_1,\ldots,\varpi_r\in\Lambda^+(\fraka)$ (see \eqref{eq:DefLambdaG+a} for the definition of $\Lambda^+(\fraka)$). Using Lemma~\ref{lem:BoundednessAFct}~(1) we define real analytic functions $u_1,\ldots,u_{2r}$ on $K$ by the formulas
\begin{equation*}
 u_j(k) := a(\tilde{w}_0^{-1}k^{-1})^{\varpi_j}, \qquad u_{r+j}(k) := a(\tilde{w}_0^{-1}k^{-1}\sigma(k))^{\varpi_j}
\end{equation*}
for $j=1,\ldots,r$. Note that the functions $u_j$ are non-negative on $K$ since the functions $a^{\varpi_j}$ are positive on the dense subset $K\cap N\tilde{w}_0MAN\subseteq K$. Then the distribution $\tilde{K}_{\alpha,\beta}$ can be written as
\begin{equation*}
 \tilde{K}_{\alpha,\beta} = u_1^{s_1}\cdots u_{2r}^{s_{2r}}
\end{equation*}
with $\alpha=\sum_{j=1}^rs_j\varpi_j$, $\beta=\sum_{j=1}^rs_{r+j}\varpi_j$ and Theorem~\ref{thm:KashiwaraKawai} yields the claim.\qedhere
\end{enumerate}
\end{proof}

We now use Theorem~\ref{thm:MeromorphicPropertiesKtilde} to construct intertwining operators $\pi_\nu|_H\to \tau_{\nu'}$. For the statement let $L(E,F)$ denote the space of bounded linear operators between two Fr\'{e}chet spaces $E$ and $F$ endowed with the topology of pointwise convergence.

\begin{theorem}\label{thm:ConvergenceMeromorphicContinuationIntertwiners}
\begin{enumerate}
\item For $\alpha,\beta\in\fraka_\CC^*$ with $\Re\alpha,\Re\beta\in\fraka_+^*$ the integral in \eqref{eq:DefIntertwinerCptPicture} converges absolutely for every $f\in C^\infty(X)$ and defines a function $A(\alpha,\beta)f\in C^\infty(X_H)$ depending holomorphically on $\alpha,\beta\in\fraka_{+,\reg}^*+i\fraka^*$.
\item The family of operators $A(\alpha,\beta):C^\infty(X)\to C^\infty(X_H)$ extends meromorphically in $\alpha,\beta\in\fraka_\CC^*$. More precisely, there exist $X_1,\ldots,X_M$, $Y_1,\ldots,Y_M\in\fraka$ and $d_1,\ldots,d_M\in\ZZ$ such that the map
\begin{equation*}
 (\alpha,\beta)\mapsto\prod_{j=1}^M\Gamma(\alpha(X_j)+\beta(Y_j)+d_j)^{-1} \cdot A(\alpha,\beta)
\end{equation*}
extends to a non-trivial entire function $\fraka_\CC^*\times\fraka_\CC^*\to L(C^\infty(X),C^\infty(X_H))$.
\item Let $(\alpha,\beta)\in\fraka_\CC^*\oplus\fraka_\CC^*$ be a regular point of $A(\alpha,\beta)$. Then for
\begin{equation}
 \nu=-w_0\alpha+\sigma\beta-w_0\beta+\rho, \qquad \nu'=-\alpha|_{\fraka^\sigma_\CC}-\rho_H\label{eq:ParameterRelation}
\end{equation}
the map $A(\alpha,\beta)$ defines an $H$-intertwining operator $\pi_\nu|_H\to\tau_{\nu'}$, i.e.
$$ A(\alpha,\beta)\in\Hom_H(\pi_\nu|_H,\tau_{\nu'}). $$
\end{enumerate}
\end{theorem}

\begin{proof}
By \eqref{eq:DefIntertwinerRewritten} we can write $A(\alpha,\beta)$ as
$$ A(\alpha,\beta)f(k_H) = \langle\tilde{K}_{\alpha,\beta},f(k_H\blank)\rangle. $$
Since the map $K_H\times C^\infty(K)\to C^\infty(K),\,(k_H,f)\mapsto f(k_H\blank)$ is smooth the statement in (1) is clear by Theorem~\ref{thm:MeromorphicPropertiesKtilde}~(1). For part (2) let
$$ \gamma(\alpha,\beta) = \prod_{j=1}^M\Gamma(\alpha(X_j)+\beta(Y_j)+d_j)^{-1} $$
with $X_1,\ldots,X_M$, $Y_1,\ldots,Y_M\in\fraka$ and $d_1,\ldots,d_M\in\ZZ$ such that $(\alpha,\beta)\mapsto\gamma(\alpha,\beta)\tilde{K}_{\alpha,\beta}$ is an entire function $\fraka_\CC^*\oplus\fraka_\CC^*\to\calD'(K)$ (see Theorem~\ref{thm:MeromorphicPropertiesKtilde}~(2)). We show that $\gamma(\alpha,\beta)A(\alpha,\beta)$ is holomorphic in $\alpha,\beta\in\fraka_\CC^*$ with values in $L(C^\infty(X),C^\infty(X_H))$. Note that this function is holomorphic if and only if for every $f\in C^\infty(X)$ the map $(\alpha,\beta)\mapsto\gamma(\alpha,\beta)A(\alpha,\beta)f$ is holomorphic on $\fraka_\CC^*\oplus\fraka_\CC^*$ with values in $C^\infty(X_H)$. To prove this it is enough to see that for any $\phi\in C^\infty(X_H)'=\calD'(X_H)$ the scalar function $(\alpha,\beta)\mapsto\gamma(\alpha,\beta)\langle\phi,A(\alpha,\beta)f\rangle$ is meromorphic because weakly holomorphic implies holomorphic. Define $g(k,k_H)=f(k_H k)$ which is a smooth function on $K \times K_H$. Then $\langle\phi,g\rangle$ is a smooth function on $K$ and we have
$$ \gamma(\alpha,\beta)\langle\phi,A(\alpha,\beta)f\rangle = \gamma(\alpha,\beta)\langle\tilde{K}_{\alpha,\beta},\langle\phi,g\rangle\rangle $$
which is holomorphic by the choice of $\gamma(\alpha,\beta)$. This proves statement (2). For statement (3) we have to show
$$ A(\alpha,\beta)\circ\pi_\nu(h) = \tau_{\nu'}(h)\circ A(\alpha,\beta) \qquad \forall\,h\in H. $$
with $\nu=\nu(\alpha,\beta)$ and $\nu'=\nu'(\alpha,\beta)$. Since this identity is meromorphic in $\alpha,\beta\in\fraka_\CC^*$ it suffices to show it for $\Re\alpha,\Re\beta\in\fraka_\CC^+$ where the integral converges absolutely. In this case we have, using formulas \eqref{eq:IntFormulaGActionOnK}, \eqref{eq:ActionOnCinftyK/M} and Proposition~\ref{prop:KernelEquivariance}
\begin{align*}
 & (A(\alpha,\beta)\pi_\nu(h)f)(k_H)\\
 ={}& \int_K K_{\alpha,\beta}(k,k_H)e^{-(\nu+\rho)H(h^{-1}k)}f(\kappa(h^{-1}k)) \td k\\
 ={}& \int_K K_{\alpha,\beta}(h^{-1}k,h^{-1}k_H)e^{-(\nu+\rho)H(h^{-1}k)}f(\kappa(h^{-1}k)) \td k\\
 ={}& \int_K e^{(-w_0\alpha+\sigma\beta-w_0\beta)H(h^{-1}k)}e^{\alpha H(h^{-1}k_H)}K_{\alpha,\beta}(\kappa(h^{-1}k),\kappa(h^{-1}k_H))\\
 & \hspace{7cm}e^{-(\nu+\rho)H(h^{-1}k)}f(\kappa(h^{-1}k)) \td k\\
 ={}& e^{\alpha H(h^{-1}k_H)} \int_K K_{\alpha,\beta}(\kappa(h^{-1}k),\kappa(h^{-1}k_H))f(\kappa(h^{-1}k))e^{-2\rho H(h^{-1}k)} \td k\\
 ={}& e^{\alpha H(h^{-1}k_H)} \int_K K_{\alpha,\beta}(k,\kappa(h^{-1}k_H))f(k) \td k\\
 ={}& e^{\alpha H(h^{-1}k_H)} (A(\alpha,\beta)f)(\kappa(h^{-1}k_H)).
\end{align*}
Here we used for the third equality that $h^{-1}k_H\in H$ decomposes according to the decomposition $H=K_HM_HA_HN_H$ into
$$ h^{-1}k_H = \kappa(h^{-1}k_H)m_Hae^{H(h^{-1}k_H)}n_H $$
with $m_H\in M_H$, $a\in A_H\cap M$ and $n_H\in N_H$. This further implies that
$$ (\tau_{\nu'}(h)A(\alpha,\beta)f)(k_H) = e^{-(\nu'+\rho_H)H(h^{-1}k_H)}a^{-(\nu'+\rho_H)}f(\kappa(h^{-1}k_H)). $$
By the definition of $\nu'$ we have $(\nu'+\rho_H)|_{\fraka_H\cap\frakm}=0$ and hence $a^{-(\nu'+\rho_H)}=\1$ and the intertwining identity follows.
\end{proof}

\begin{remark}
The result in Theorem~\ref{thm:ConvergenceMeromorphicContinuationIntertwiners} is abstract and does not provide any information about the location of the poles of $A(\alpha,\beta)$ and their nature. In Section~\ref{sec:BernsteinSatoIdentities} we outline a possible method to study the residues.
\end{remark}

Recall the identifications $I^G(\nu)\cong C^\infty(X)$ and $I^H(\nu')\cong C^\infty(X_H)$.

\begin{corollary}\label{cor:IntertwinersNoncptPicture}
For $\Re\alpha,\Re\beta\in\fraka_+^*$ and $\nu,\nu'$ as in \eqref{eq:ParameterRelation} the intertwining operator $\tilde{A}(\alpha,\beta):I^G(\nu)\to I^H(\nu')$ corresponding to $A(\alpha,\beta):C^\infty(X)\to C^\infty(X_H)$ is given by the convergent integrals
\[ A(\alpha,\beta)f(h) = \int_K K_{\alpha,\beta}(k,h)f(k) \td k = \int_{\overline{N}} K_{\alpha,\beta}(\overline{n},h)f(\overline{n}) \td\overline{n}, \qquad h\in H. \]
\end{corollary}

\begin{proof}
The unique extension of a function $\varphi\in C^\infty(X_H)$ to $\tilde{\varphi}\in I^H(\nu')$ is given by
\[ \tilde{\varphi}(k_Hm_Ha_Hn_H) = a_H^{-\nu'-\rho_H}\varphi(k_H). \]
Therefore we obtain by Proposition~\ref{prop:KernelEquivariance}~(2) for $f\in I^G(\nu)$:
\begin{align*}
 \tilde{A}(\alpha,\beta)f(k_Hm_Ha_Hn_H) &= a_H^{-\nu'-\rho_H} A(\alpha,\beta)(f|_K)(k_H)\\
 &= a_H^\alpha \int_K K_{\alpha,\beta}(k,k_H)f(k) \td k\\
 &= \int_K K_{\alpha,\beta}(k,k_Hm_Ha_Hn_H)f(k) \td k.
\end{align*}
Further, using the integral formula~\eqref{eq:IntFormulaKNbar} we obtain
\begin{align*}
 \tilde{A}(\alpha,\beta)f(h) &= \int_{\overline{N}} K_{\alpha,\beta}(\kappa(\overline{n}),h)f(\kappa(\overline{n})) e^{-2\rho H(\overline{n})} \td\overline{n}\\
 &= \int_{\overline{N}} e^{(w_0\alpha-\sigma\beta+w_0\beta)H(\overline{n})}K_{\alpha,\beta}(\overline{n},h) e^{(\nu+\rho)H(\overline{n})}f(\overline{n}) e^{-2\rho H(\overline{n})} \td\overline{n}\\
 &= \int_{\overline{N}} K_{\alpha,\beta}(\overline{n},h)f(\overline{n}) \td\overline{n},
\end{align*}
finishing the proof.
\end{proof}

\begin{remark}
Using Corollary~\ref{cor:IntertwinersNoncptPicture} the intertwining operators $\tilde{A}(\alpha,\beta)$ can be studied in the \textit{non-compact picture}. The non-compact picture is obtained by restricting functions in $I^G(\nu)$ and $I^H(\nu')$ to $\overline{N}$ and $\overline{N}_H$, respectively. Then $\tilde{A}(\alpha,\beta)$ is an integral operator on flat space with kernel $K_{\alpha,\beta}(\overline{n},\overline{n}_H)$ on $N\times N_H$. In the case $(G,H)=(\upO(1,n),\upO(1,n-1))$ these operators were investigated earlier by Kobayashi--Speh~\cite{KS,KS14} (see also \cite{Kob13,MO12}).
\end{remark}

\begin{remark}\label{rem:InvDistVectors}
By the Schwartz Kernel Theorem the intertwining operators $A(\alpha,\beta):C^\infty(X)\to C^\infty(X_H)$ are given by distribution kernels in $\calD'(X\times X_H)$. For $\Re\alpha,\Re\beta\in\fraka_+^*$ these kernels are precisely $K_{\alpha,\beta}|_{K\times K_H}$. Abusing notation we also write $K_{\alpha,\beta}$ for their meromorphic extension in $\alpha,\beta\in\fraka_\CC^*$. With $\nu,\nu'$ as in \eqref{eq:ParameterRelation} these distributions are $\Delta(H)$-invariant distribution vectors for the representations $\pi_\nu\otimes\tau_{\nu'}$ of $G\times H$ on $C^\infty(X\times X_H)$, i.e.
\[ K_{\alpha,\beta}\in(\pi_\nu\otimes\tau_{\nu'})^{-\infty,\Delta(H)}. \]
\end{remark}

\subsection{Induction parameters}\label{sec:InductionParameters}

We study the relation \eqref{eq:ParameterRelation} between the parameters $\alpha,\beta\in\fraka_\CC^*$ of the kernel $K_{\alpha,\beta}$ and the induction parameters $\nu\in\fraka_\CC^*$ and $\nu'\in(\fraka_H)_\CC^*$.

First we note that the condition $\nu'=-\alpha|_{\fraka^\sigma_\CC}-\rho_H$ means that $\nu'\equiv-\alpha-\rho_H$ on $\fraka^\sigma=\fraka_H\cap\fraka$ and $\nu'\equiv-\rho_H$ on $\fraka_H\cap\frakm$ in the decomposition $\fraka_H=(\fraka_H\cap\fraka)\oplus(\fraka_H\cap\frakm)$. This gives the necessary condition
\begin{equation*}
 (\nu'+\rho_H)|_{\fraka_H\cap\frakm} = 0.
\end{equation*}

By Lemma~\ref{lem:SigmaW0Commute} we obtain the joint eigenspace decomposition for $w_0$ and $\sigma$:
\[ \fraka = (\fraka^\sigma\cap\fraka^{-w_0})\oplus(\fraka^\sigma\cap\fraka^{w_0})\oplus(\fraka^{-\sigma}\cap\fraka^{-w_0})\oplus(\fraka^{-\sigma}\cap\fraka^{w_0}), \]
where $\fraka^{\pm w_0}=\{H\in\fraka:w_0H=\pm H\}$. 

We first consider the case where $w_0|_\fraka=-1$.
The relation \eqref{eq:ParameterRelation} between $\alpha,\beta$ and $\nu,\nu'$ then reads
\[ \nu = \alpha+\beta+\sigma\beta+\rho, \qquad \nu' = -\alpha|_{\fraka^\sigma_\CC}-\rho_H. \]
By Remark~\ref{rem:KernelDependenceOnBeta} the kernel $K_{\alpha,\beta}(g,h)$ does not depend on the values of $\beta$ on $\fraka_\CC^{-\sigma}$, and neither do $\nu$ and $\nu'$. Hence we may assume $\beta|_{\fraka^{-\sigma}}=0$, i.e.\ $\beta\in(\fraka^\sigma)_\CC^*$ and the relations read
\[ \nu = \alpha+2\beta+\rho, \qquad \nu' = -\alpha|_{\fraka^\sigma_\CC}-\rho_H. \]
The following result is immediate:

\begin{lemma}
Assume that $w_0|_\fraka=-1$. Then the map
\[ (\alpha,\beta)\mapsto(\nu,\nu')=(\alpha+2\beta+\rho,-\alpha|_{\fraka^\sigma_\CC}-\rho_H) \]
defines a bijection
\[ \fraka_\CC^*\oplus(\fraka^\sigma)_\CC^* \to \{(\nu,\nu')\in\fraka_\CC^*\oplus(\fraka_H)_\CC^*:(\nu'+\rho_H)|_{\fraka_H\cap\frakm}=0\}. \]
\end{lemma}

This gives the largest possible set of induction parameters $\nu$ and $\nu'$ that can be treated with the kernel $K_{\alpha,\beta}(g,h)$.

Returning to the general case, we may assume that $\beta=0$ on $\fraka^{w_0\sigma}$ by Remark~\ref{rem:KernelDependenceOnBeta}.  Then we have

\begin{lemma}\label{lem:Parameters2}
The map
\begin{multline*}
 \fraka_\CC^*\oplus(\fraka^{-w_0\sigma})_\CC^* \to \{(\nu,\nu')\in\fraka_\CC^*\oplus(\fraka_H)_\CC^*:(\nu'+\rho_H)|_{(\fraka_H\cap\frakm)}=0,\\
 (\nu'+\rho_H)|_{\fraka^\sigma\cap\fraka^{w_0}}=(\nu-\rho)|_{\fraka^\sigma\cap\fraka^{w_0}}\}
\end{multline*}
given by
\[ (\alpha,\beta) \mapsto (\nu,\nu')=(-w_0\alpha+2\sigma\beta+\rho,-\alpha|_{\fraka^\sigma_\CC}-\rho_H) \]
is surjective and has fibers
\[ (\alpha_0,\beta_0) + \{(2\gamma,-\gamma):\gamma\in(\fraka^{-\sigma}\cap\fraka^{w_0})_\CC^*\}. \]
\end{lemma}

Again this gives the largest possible set of induction parameters $\nu$ and $\nu'$ that can be treated with the kernel $K_{\alpha,\beta}(g,h)$. However, since the map $(\alpha,\beta)\mapsto(\nu,\nu')$ is not necessarily injective there might be several different integral kernels $K_{\alpha,\beta}(g,h)$ which define intertwining operators on the same representations $I^G(\nu)\times I^H(\nu')$.
%We cannot expect uniqueness of invariant bilinear forms in these cases. 
%This corresponds to the fact that the condition on $(\nu,\nu')$ in Lemma~\ref{lem:Parameters2} is not generic, so it is still possible to have generic uniqueness. 
%Further, since the kernel $K_{\alpha,\beta}(g,h)$ is independent of the values of $\beta$ on $\fraka^{w_0\sigma}$ the same kernel might give rise to invariant bilinear forms for different induction parameters $(\nu,\nu')$.

\subsection{Uniqueness}\label{sec:Uniqueness}

In this section we outline a method to study generic bounds for the dimension of the space of intertwining operators in the case where there exists an open orbit of $\Delta(H)$ on $G/P\times H/P_H$. This method is applied in Section~\ref{sec:RankOne} to the rank one examples from Section~\ref{sec:ExRankOne}. We expect this method to work also in other cases.

\subsubsection{Invariant distributions}\label{sec:InvariantDistributions}

Note that the non-degenerate invariant bilinear form
$$ I^H(\nu')\times I^H(-\nu')\to\CC, \quad (f_1,f_2)\mapsto\int_{K_H}f_1(k)f_2(k)\td k $$
induces an $H$-invariant embedding $I^H(\nu')\to I^H(-\nu')'$ into the dual representation $I^H(-\nu')'$. This allows us to view each intertwining operator $A:I^G(\nu)\to I^H(\nu')$ as an intertwiner $I^G(\nu)\to I^H(-\nu')'$. The method we present even works for intertwining operators $A:I^G(\nu)\to I^H(-\nu')'$.

To every intertwining operator $A:I^G(\nu)\to I^H(-\nu')'$ we associate a distribution $K\in\calD'(G\times H)$ which is equivariant under the action of the group $L=H\times P\times P_H$ where $H$ acts diagonally by left-multiplication and $P\times P_H$ by right-multiplication. For this we need to associated to each test function in $C_c^\infty(G)$ and $C_c^\infty(H)$ a function in $I^G(\nu)$ and $I^H(\nu')$, respectively.

Let $\td p$ be a left-invariant measure on $P$ and let $\Delta_P$ be the modular function for $P$, i.e.
\begin{equation*}
 \int_P f(pp')\td p = \Delta_P(p')^{-1}\int_P f(p)\td p, \qquad p'\in P.
\end{equation*}
Then $\Delta_P$ can be written as $\Delta_P(p)=|\det\Ad(p^{-1})|=a^{-2\rho}$ for $p=man\in MAN$. Let $\varpi$ be a character of $P$. For $f\in C_c^\infty(G)$ put
\begin{equation}
 \flat f(g) := \int_P \Delta_P(p)^{-\frac{1}{2}}\varpi(p)f(gp)\td p.\label{eq:DefFlat}
\end{equation}
Denote by $r(p)$ the right-regular representation of $p\in P$ on $C^\infty(G)$, i.e.\ $r(p)f(g)=f(gp)$.

\begin{lemma}[{\cite[Lemma 4.6]{KV96}}]\label{lem:FlatOperator}
The identity \eqref{eq:DefFlat} defines a surjective continuous linear operator $\flat:C_c^\infty(G)\to\Ind_P^G(\varpi)$ which is $G$-equivariant and satisfies
\begin{equation}
 \flat\circ r(p) = \Delta_P(p)^{-\frac{1}{2}}\varpi(p)^{-1}\flat, \qquad p\in P.\label{eq:EquivarianceFlatB}
\end{equation}
\end{lemma}

Denote by $\flat_G$ and $\flat_H$ the corresponding operators for $G$ and $H$ with induction parameters $\varpi=e^\nu$ and $\varpi=e^{-\nu'}$, respectively. We note that the transpose $\flat_H^t$ defines an $H$-equivariant injective continuous linear operator $I^H(\nu')'\to\calD'(H)$.

Now let $A:I^G(\nu)\to I^H(-\nu')'$ be a continuous intertwining operator for the action of $H$. Then the composition $\flat_H^t\circ A\circ\flat_G$ is a continuous linear operator $C_c^\infty(G)\to\calD'(H)$ and hence, thanks to the Schwartz Kernel Theorem, given by a distribution kernel $K_A\in\calD'(G\times H)$, i.e.
$$ \int_H (A\flat_G\phi)(h)(\flat_H\psi)(h) \td h = \langle K_A, \phi\otimes\psi\rangle, \qquad \phi\in C_c^\infty(G),\psi\in C_c^\infty(H). $$
Then $K_A$ is $H$-left invariant and by \eqref{eq:EquivarianceFlatB} it transforms under the right-action of $P\times P_H$ by
\begin{equation*}
 r(man,m_Ha_Hn_H)K_A = a^{\nu-\rho}a_H^{-\nu'-\rho_H}K_A.
\end{equation*}

\begin{remark}
Note that for $A=\tilde{A}(\alpha,\beta):I^G(\nu)\to I^H(\nu')\subseteq I^H(-\nu')'$ as in Corollary~\ref{cor:IntertwinersNoncptPicture} we have $K_A=K_{\alpha,\beta}$, the kernel defined in Section~\ref{sec:InvriantKernel}.
\end{remark}

Let $L=H\times P\times P_H$ act on $G\times H$ by $(h',p,p_H)\cdot(g,h)=(h'gp^{-1},h'hp_H^{-1})$. Define a character $\beta$ of $L$ by $\beta(h,man,m_Ha_Hn_H):=a^{\nu-\rho}a_H^{-\nu'-\rho_H}$. Let $E=\CC_\beta$ be the corresponding one-dimensional representation of $L$ on $\CC$. The group $L$ acts on $C_c^\infty(G\times H,E)$ by
\begin{equation*}
 x\cdot\varphi(y)=\beta(x)\varphi(x^{-1}\cdot y)=\beta(x)(x\cdot\varphi)(y).
\end{equation*}
By duality we obtain an action of $L$ on $\calD'(G\times H,E)=C_c^\infty(G\times H,E)'$ which is given by $\langle x\cdot u,\varphi\rangle=\langle u,x^{-1}\cdot\varphi\rangle$. Then our calculation above shows that the distribution $K_A$ arising from the intertwining operator $A$ is invariant under $L$. Denote the space of $L$-invariant distributions on $G\times H$ with values in $E$ by $\calD'(G\times H,E)^L$.

\begin{lemma}\label{lem:MapFormsToDistributions}
The map $A\mapsto K_A$ sending a continuous $H$-intertwining operator $A:I^G(\nu)\times I^H(-\nu')'$ to the invariant distribution $K_A\in\calD'(G\times H,E)^L$ is injective.
\end{lemma}

This allows us to study uniqueness of intertwining operators by studying uniqueness of invariant distributions.

\begin{remark}
It would be more natural to consider $H$-invariant distributions on $G/P\times H/P_H$ with values in a line bundle instead of $L$-invariant distributions on $G\times H$ with values in $E$. However, to be able to directly apply the results in \cite{KV96} we use the latter formulation.
\end{remark}

\subsubsection{Bruhat's Vanishing Theorem}

For the uniqueness question of invariant distributions it is essential to have an open orbit of $H$ on the double flag variety $G/P\times H/P_H$. On each open orbit we can use the following classical fact (see e.g.\ \cite[Lemma 3.2]{Dei06}):

\begin{fact}\label{fct:InvDistributionsOnHomSpaces}
Let $X=H/S$ be an $H$-homogeneous space and $\calE=H\times_SE$ a smooth $H$-homogeneous vector bundle. Then every $H$-invariant distribution $u\in\calD'(X,\calE)^H$ is given by a smooth $H$-invariant section of the dual bundle $\calE^*$. Moreover, we have $\dim\calD'(X,\calE)^H=\dim (E^*)^S$. In particular, if $\calE$ is irreducible, it only has non-trivial invariant sections if it is trivial and in this case the space of invariant sections is one-dimensional.
\end{fact}

This statement assures uniqueness of invariant distributions on each open $H$-orbit in $G/P\times H/P_H$. However, there are also smaller $H$-orbits in the double flag variety on which more singular invariant distributions can be supported. To deal with these cases we use a powerful result by Bruhat (see e.g.\ \cite[Vanishing Theorem 3.15]{KV96} or \cite[Section 5.2]{War72}):

\begin{theorem}[Bruhat]\label{thm:KV}
Let $X$ be an $L$-space and $\calO\subseteq X$ a closed $L$-orbit, $\calO=L\cdot x_0\cong L/S$. Let $\beta$ be a differentiable representation of $L$ on a Fr\'{e}chet space $E$. Denote by $\calD'_\calO(X,E)$ the distributions on $X$ with values in $E$ which are supported on $\calO$ and by $\calD'_\calO(X,E)^L$ its invariants. Further, let $V=(T_{x_0}X)/(T_{x_0}\calO)$, the quotient of the tangent spaces at $x_0$ in $X$ and $\calO$, endowed with the natural action of $S$, and consider the symmetric powers $S^r(V_\CC)$ as $S$-representations. Finally let $\chi$ be the character of $S$ given by $\chi(g)=\frac{|\det\Ad_L(g)|}{|\det\Ad_S(g)|}$, the quotient of the modular functions of $L$ and $S$. Then
\begin{equation}
 \Hom_S(E\otimes\CC_\chi,S^r(V_\CC)) = 0 \quad \forall\,r\in\NN_0 \qquad \Rightarrow \qquad \calD'_\calO(X,E)^L = 0.\label{eq:VanishingThmCondition}
\end{equation}
\end{theorem}

Given a stratification $\calO_1,\ldots,\calO_r$ of $L$-orbits with $\calO_r$ open one can now proceed by induction. Putting $X=\calO_{r-k}\cup\ldots\cup\calO_r$, $k=0,\ldots,r$, and assuming that a distribution vanishes on $\calO_{r-k+1}\cup\ldots\cup\calO_r$ we can use Theorem~\ref{thm:KV} to find a condition for the vanishing of the distribution on $\calO_{r-k}$.

\subsection{Bernstein--Sato identities}\label{sec:BernsteinSatoIdentities}

Although Theorem~\ref{thm:ConvergenceMeromorphicContinuationIntertwiners} proves meromorphic continuation of the intertwining operators $A(\alpha,\beta)$, it does not provide any information about the location and nature of the poles. Further, it is not clear in general, how to calculate the residues explicitly. In this section we outline a strategy to obtain explicit Bernstein--Sato identities which can be used to find poles and residues of the intertwining operators. This strategy generalizes a method used by Beckmann--Clerc~\cite{BC12} for the product situation $G=\SO(1,n)\times\SO(1,n)$, $H=\Delta\SO(1,n)$ the diagonal.

We require that the space of $H$-intertwining operators $I^G(\nu)\to I^H(-\nu')'$ has generically dimension equal to the number of open $H$-orbits on $G/P\times H/P_H$. For simplicity of exposition we assume that there is only one open orbit. This property is crucial in what follows.

\subsubsection{Intertwining operators vs.\ invariant bilinear forms}

We first note that continuous $H$-intertwining operators $A:I^G(\nu)\to I^H(-\nu')'$ are in one-to-one correspondence with continuous $H$-invariant bilinear forms $(\blank,\blank)_A:I^G(\nu)\times I^H(-\nu')\to\CC$ in the sense that
$$ (\phi,\psi)_A = \langle A\phi,\psi\rangle, \qquad \phi\in I^G(\nu),\psi\in I^H(-\nu'). $$
Each such invariant bilinear form can by continuity be extended to an $\Delta(H)$-invariant linear functional on
\begin{align*}
 I(\nu,-\nu') :={}& \{f\in C^\infty(G\times H):f(gman,hm_Ha_Hn_H)=a^{-\nu-\rho}a_H^{\nu'-\rho_H}f(g,h)\\
 & \hspace{1.5cm}\forall\,g\in G,h\in H,man\in MAN,m_Ha_Hn_H\in M_HA_HN_H\}
\end{align*}
where we can interpret $I(\nu,-\nu')$ as the completion of the tensor product $I^G(\nu)\otimes I^H(-\nu')$ carrying the left-regular representation of $G\times H$. For $\alpha,\beta\in\fraka_\CC^*$ and $\nu,\nu'$ satisfying \eqref{eq:ParameterRelation} denote by $\ell(\alpha,\beta)$ the $H$-invariant continuous linear functional on $I(\nu,-\nu')$ corresponding to the intertwining operator $A=\tilde{A}(\alpha,\beta)$. We fix $\nu$ and $\nu'$ for the rest of this section. Note that the functional $\ell(\alpha,\beta)$ is given by the kernel $K_{\alpha,\beta}\in\calD'(G\times H)$ viewed as a distribution on $G\times H$ and its meromorphic continuation (see Section~\ref{sec:InvariantDistributions}).

\begin{lemma}
For all $\lambda,\mu\in\Lambda^+(\fraka)$ the multiplication operator
\begin{equation*}
 M_{\lambda,\mu}:C^\infty(G\times H)\to C^\infty(G\times H), \quad M_{\lambda,\mu}f(g,h) = K_{\lambda,\mu}(g,h)f(g,h)
\end{equation*}
defines an $H$-intertwiner
\[ I(\nu,-\nu')\to I(\nu-(-w_0\lambda+\sigma\mu-w_0\mu),-\nu'-\lambda|_{(\fraka^\sigma)_\CC^*}) \]
for all $\nu\in\fraka_\CC^*$, $\nu'\in(\fraka_H)_\CC^*$.
\end{lemma}

\begin{proof}
Since $\lambda,\mu\in\Lambda^+(\fraka)$ the function $K_{\lambda,\mu}(g,h)$ is by Lemma~\ref{lem:BoundednessAFct}~(1) smooth on $G$ and hence $M_{\lambda,\mu}$ maps into $C^\infty(G\times H)$. That it actually maps $I(\nu,-\nu')$ into the right space follows from Proposition~\ref{prop:KernelEquivariance}~(2). Since $K_{\lambda,\mu}(g,h)$ is by Proposition~\ref{prop:KernelEquivariance}~(1) left-invariant under $H$ the operator $M_{\lambda,\mu}$ is also $H$-intertwining and the proof is complete.
\end{proof}

To obtain an intertwining operator $N_{\lambda,\mu}$ ``in the other direction'' we conjugate with the classical Knapp--Stein intertwiners. Let $A^G(\nu):I^G(\nu)\to I^G(w_0\nu)$ denote the (regularized) classical Knapp--Stein intertwining operator for $G$ and $A^H(-\nu'):I^H(-\nu')\to I^H(-w_0\nu')$ the one for $H$. For $\lambda,\mu\in\Lambda^+(\fraka)$ the operator
\[ N_{\lambda,\mu}(\alpha,\beta):I(\nu,-\nu')\to I(\nu+(\lambda+\mu-w_0\sigma\mu),-\nu'-w_0\lambda|_{(\fraka^\sigma)_\CC^*}) \]
given by
\begin{multline*}
 N_{\lambda,\mu}(\alpha,\beta) := (A^G(w_0\nu-(-w_0\lambda+\sigma\mu-w_0\mu))\otimes A^H(-w_0\nu'-\lambda|_{(\fraka^\sigma)_\CC^*}))\\
 \circ M_{\lambda,\mu}\circ(A^G(\nu)\otimes A^H(-\nu')).
\end{multline*}
is $H$-intertwining.

Now, on $I(\nu+(\lambda+\mu-w_0\sigma\mu),-\nu'-w_0\lambda|_{(\fraka^\sigma)_\CC^*})$ the functional $\ell(\alpha+\lambda,\beta+\mu)$ is $H$-invariant and hence the functional
\begin{equation*}
 \ell(\alpha,\beta)^{\lambda,\mu} := \ell(\alpha+\lambda,\beta+\mu)\circ N_{\lambda,\mu}(\alpha,\beta).
\end{equation*}
is $H$-invariant on $I(\nu,-\nu')$. Because of the assumed generic multiplicity one property it has to be generically proportional to $\ell(\alpha,\beta)$, i.e.\ there exists a function $b_{\lambda,\mu}(\alpha,\beta)$ such that generically the following Bernstein--Sato type identity holds:
\begin{equation}
 b_{\lambda,\mu}(\alpha,\beta) \ell(\alpha,\beta) = \ell(\alpha+\lambda,\beta+\mu)\circ N_{\lambda,\mu}(\alpha,\beta).\label{eq:BernsteinSatoIdentity}
\end{equation}
If the function $b_{\lambda,\mu}(\alpha,\beta)$ is meromorphic (e.g.\ polynomial in $\alpha$ and $\beta$) we can use the identity~\eqref{eq:BernsteinSatoIdentity} to extend $\ell(\alpha,\beta)$ meromorphically by dividing by $b_{\lambda,\mu}(\alpha,\beta)$. 

\begin{remark}
We expect that in some cases $N_{\lambda,\mu}(\alpha,\beta)$ is a differential operator on $G\times H$ and $b_{\lambda,\mu}(\alpha,\beta)$ a polynomial in $\alpha$ and $\beta$. Then the transpose of $N_{\lambda,\mu}(\alpha,\beta)$ maps the integral kernel $K_{\alpha+\lambda,\beta+\mu}(g,h)$ to a multiple of the integral kernel $K_{\alpha,\beta}(g,h)$:
\[ N_{\lambda,\mu}(\alpha,\beta)^tK_{\alpha+\lambda,\beta+\mu} = b_{\lambda,\mu}(\alpha,\beta)K_{\alpha,\beta}. \]
Knowing some residue of $K_{\alpha,\beta}(g,h)$ for some parameter $(\alpha,\beta)$ this identity can be used to find the residue at the parameter $(\alpha+\lambda,\beta+\mu)$ (see \cite{BC12} for the case $(\SO(1,n)\times\SO(1,n),\Delta\SO(1,n))$).
\end{remark}
\section{Uniqueness for rank one groups}\label{sec:RankOne}

We use the technique described in Section~\ref{sec:Uniqueness} to prove that for the cases $(G,H)=(\SU(1,n;\FF),\upS(\upU(1,m;\FF)\times\upU(n-m;\FF)))$, $\FF=\RR,\CC,\HH,\OO$, the space of continuous intertwining operators $I^G(\nu)\to I^H(\nu')$ is one-dimensional for generic parameters $\nu,\nu'$. In the special cases $\FF=\RR,\CC$ with $m=n-1$ this also follows from the multiplicity-one theorems by Sun--Zhu~\cite{SZ12}. For $\FF=\RR$ Kobayashi--Speh~\cite{KS,KS14} obtained all intertwining operators, also for singular parameters $\nu,\nu'$, finding multiplicity two in some singular cases.

We use the notation of Section~\ref{sec:ExRankOne}.

\subsection{Geometry of the double flag variety}

Using Proposition~\ref{prop:OrbitsDoubleFlag} it is easy to see that the subgroup $H$ has precisely three orbits on the double flag variety $G/P\times H/P_H$. Put $L=H\times P\times P_H$ and let $L$ act on $G\times H$ by $(h',p,p_H)\cdot(g,h)=(h'gp^{-1},h'hp_H^{-1})$. Then $G\times H=\calO_1\sqcup\calO_2\sqcup\calO_3$, where
\begin{equation*}
 \calO_1 = L\cdot(\1,\1), \qquad \calO_2 = L\cdot(\tilde{w}_0,\1), \qquad \calO_3 = L\cdot(n_0\tilde{w}_0,\1)
\end{equation*}
for some arbitrary $n_0\in N\setminus N_H$ as in Section~\ref{sec:ExRankOne}. The stabilizer subgroups are given by
\begin{align*}
 S_1 &:= \Stab_L(\1,\1) = \{(p_H,p_H,p_H):p_H\in P_H\} \cong P_H,\\
 S_2 &:= \Stab_L(\tilde{w}_0,\1) = \{(g,\tilde{w}_0^{-1}g\tilde{w}_0,g):g\in M_HA_H\} \cong M_HA_H,\\
 S_3 &:= \Stab_L(n_0\tilde{w}_0,\1) = \{(g,\tilde{w}_0^{-1}g\tilde{w}_0,g):g\in\Stab_{M_H}(n_0)\} \cong \Stab_{M_H}(n_0).
\end{align*}
The orbit $\calO_3$ is open and hence it has tangent space
\begin{equation*}
 T_{(n_0\tilde{w}_0,\1)}\calO_3 = \frakg\times\frakh.
\end{equation*}
The action of $L$ on $(\tilde{w}_0,\1)$ is given by $(h,p,p_H)\cdot(\tilde{w}_0,\1)=(h\tilde{w}_0p^{-1},hp_H^{-1})=(\tilde{w}_0(\tilde{w}_0^{-1}h\tilde{w}_0)p^{-1},hp_H^{-1})$ and therefore the tangent space of $\calO_2$ at $(\tilde{w}_0,\1)$ is given by
\begin{align*}
 T_{(\tilde{w}_0,\1)}\calO_2 &= \{(\Ad(\tilde{w}_0^{-1})X+Y,X+Z):X\in\frakh,Y\in\frakp,Z\in\frakp_H\}\\
 &= \{(\Ad(\tilde{w}_0^{-1})X+Y,X):X\in\frakh,Y\in\frakp+\overline{\frakn}_H\} = (\frakp+\overline{\frakn}_H)\times\frakh.
\end{align*}
The action of $L$ on $(\1,\1)$ is given by $(h,p,p_H)\cdot(\1,\1)=(hp^{-1},hp_H^{-1})$ and therefore the tangent space of $\calO_1$ at $(\1,\1)$ is given by
\begin{align*}
 T_{(\1,\1)}\calO_1 &= \{(X+Y,X+Z):X\in\frakh,Y\in\frakp,Z\in\frakp_H\}\\
 &= \{(X+Y,X+Z):X\in\overline{\frakn}_H,Y\in\frakp,Z\in\frakp_H\}.
\end{align*}
Thus we find the quotients (identified with subspaces of $\frakg\oplus\frakh$)
\begin{align*}
 (T_{(\tilde{w}_0,\1)}(G\times H))/(T_{(\tilde{w}_0,\1)}\calO_2) &\cong  \overline{\frakn}^{-\sigma}\times\{0\},\\
 (T_{(\1,\1)}(G\times H))/(T_{(\1,\1)}\calO_1) &\cong \{(X+Y,-X):X\in\overline{\frakn}_H,Y\in\overline{\frakn}^{-\sigma}\}.
\end{align*}
An element $(\tilde{w}_0g\tilde{w}_0^{-1},g,\tilde{w}_0g\tilde{w}_0^{-1})\in S_2$, $g\in M_HA_H$, acts on the quotient space $\overline{\frakn}^{-\sigma}\times\{0\}$ of the tangent spaces at $(\tilde{w}_0,\1)$ as follows:
\begin{align*}
 (\tilde{w}_0g\tilde{w}_0^{-1},g,\tilde{w}_0g\tilde{w}_0^{-1})\cdot(X,0) &= \left.\frac{\td}{\td t}\right|_{t=0} (\tilde{w}_0g\tilde{w}_0^{-1},g,\tilde{w}_0g\tilde{w}_0^{-1})\cdot(\tilde{w}_0,\1)e^{t(X,0)}\\
 &= \left.\frac{\td}{\td t}\right|_{t=0} (\tilde{w}_0e^{t\Ad(g)X},\1)\\
 &= (\Ad(g)X,0).
\end{align*}
An element $(p_H,p_H,p_H)\in S_3$, $p_H\in P_H$, acts on the quotient space $\{(X+Y,-X):X\in\overline{\frakn}_H,Y\in\overline{\frakn}^{-\sigma}\}$ of the tangent spaces at $(\1,\1)$ by
\begin{align*}
 (p_H,p_H,p_H)\cdot(X+Y,-X) &= \left.\frac{\td}{\td t}\right|_{t=0} (p_H,p_H,p_H)\cdot(\1,\1)e^{t(X+Y,-X)}\\
 &= \left.\frac{\td}{\td t}\right|_{t=0} (p_He^{t(X+Y)}p_H^{-1},p_He^{-tX}p_H^{-1})\\
 &= (\pr_{\overline{\frakn}}\Ad(p_H)(X+Y),-\pr_{\overline{\frakn}}\Ad(p_H)X),
\end{align*}
where $\pr_{\overline{\frakn}}:\frakg=\frakp+\overline{\frakn}\to\overline{\frakn}$ denotes the projection.

\subsection{Generic uniqueness}

We now show that generically $\dim\calD'(G\times H,E)^L=1$. For this we investigate the condition~\eqref{eq:VanishingThmCondition} for the orbits $\calO_1$ and $\calO_2$.

\subsubsection{Distributions supported on $\calO_2$}\label{sec:Rk1UniquenessO2}

Let $X=\calO_3\cup\calO_2\subseteq G\times H$ open, $\calO=\calO_2\subseteq X$ closed. An invariant distribution $u\in\calD'(G\times H,E)^L$ which vanishes on $\calO_3$ has support on $\overline{\calO_2}$. Hence the restriction $u|_X\in\calD'_\calO(X,E)^L$. By Theorem~\ref{thm:KV} the restriction $u|_X$ has to be zero if the condition~\eqref{eq:VanishingThmCondition} is fulfilled. We identify $S=S_2\cong M_HA_H$ by $(\tilde{w}_0g\tilde{w}_0^{-1},g,\tilde{w}_0g\tilde{w}_0^{-1})\mapsto g$. Then the character $\beta$ takes the form
\begin{equation*}
 \beta(ma) = a^{\nu-\rho}(a^{-1})^{-\nu'-\rho_H} = a^{(\nu-\rho)+(\nu'+\rho_H)}, \qquad ma\in M_HA_H.
\end{equation*}
The character $\chi$ is given by
\begin{equation*}
 \chi(ma) = \frac{|\det\Ad_L(\tilde{w}_0ma\tilde{w}_0^{-1},ma,\tilde{w}_0ma\tilde{w}_0^{-1})|}{|\det\Ad_{M_HA_H}(ma)|} = a^{2\rho-2\rho_H}.
\end{equation*}
Hence $E\otimes\CC_\chi$ is the representation of $M_HA_H$ given by the character $ma\mapsto a^{(\nu+\rho)+(\nu'-\rho_H)}$.

\begin{enumerate}
\item For $\FF=\RR,\CC,\HH$ we have $M_H=\upS(\Delta\upU(1;\FF)\times\upU(m-1;\FF)\times\upU(n-m;\FF))$. The adjoint action of $M_H$ on $\overline{\frakn}^{-\sigma}\cong\FF^{n-m}$ is given by the defining representation of the factor $\upU(n-m;\FF)$ on $\FF^{n-m}$ and the action of $\upU(1;\FF)$ on $\FF^{n-m}$ by multiplication from the right. Therefore, there can only be a non-trivial $M_H$-invariant in $S^r(\overline{\frakn}^{-\sigma}_\CC)$ if $r$ is even. Further, the adjoint action of $A_H$ on $\overline{\frakn}^{-\sigma}\subseteq\frakg_{-\alpha}$ is $\Ad(e^{tH_0})X=e^{-t}X$ and hence $e^{tH_0}$ acts on $S^r(\overline{\frakn}^{-\sigma}_\CC)$ by $e^{-rt}$. Therefore, Theorem~\ref{thm:KV} implies:
\begin{equation*}
 (\nu+\rho)+(\nu'-\rho_H)\notin(-2\NN_0) \qquad \Rightarrow \qquad \calD'_{\calO_2}(\calO_2\cup\calO_3,E)^L=0.
\end{equation*}
\item For $\FF=\OO$ we have $M_H=\Spin(7)$ acting on $\overline{\frakn}^{-\sigma}\cong\OO\cong\RR^8$ by the spin representation. Since $\SU(4)\cong\Spin(6)\subseteq\Spin(7)$ acts on $\RR^8\cong\CC^4$ by the defining representation, there can again only be a non-trivial $M_H$-invariant in $S^r(\overline{\frakn}^{-\sigma}_\CC)$ if $r$ is even. The adjoint action of $A_H$ on $\overline{\frakn}^{-\sigma}=\frakg_{-\alpha}$ is $\Ad(e^{tH_0})X=e^{-t}X$ and hence $e^{tH_0}$ acts on $S^r(\overline{\frakn}^{-\sigma})$ by $e^{-rt}$. Therefore, Theorem~\ref{thm:KV} implies:
\begin{equation*}
 (\nu+\rho)+(\nu'-\rho_H)\notin(-2\NN_0) \qquad \Rightarrow \qquad \calD'_{\calO_2}(\calO_2\cup\calO_3,E)^L=0.
\end{equation*}
\end{enumerate}

\subsubsection{Distributions supported on $\calO_1$}

Let $X=G\times H$, $\calO=\calO_1\subseteq X$ closed. An invariant distribution $u\in\calD'(G\times H,E)^L$ which vanishes on the open subset $\calO_2\cup\calO_3\subseteq X$ has support on $\calO_1$, hence $u\in\calD'_\calO(X,E)^L$. Again we check condition~\ref{eq:VanishingThmCondition}. We identify $S=S_1\cong P_H$ by $(p_H,p_H,p_H)\mapsto p_H$. The character $\beta$ takes the form
\begin{equation*}
 \beta(man) = a^{\nu-\rho}a^{-\nu'-\rho_H} = a^{(\nu-\rho)-(\nu'+\rho_H)}, \qquad man\in M_HA_HN_H.
\end{equation*}
The character $\chi$ is given by
\begin{equation*}
 \chi(man) = \frac{|\det\Ad_L(man,man,man)|}{|\det\Ad_{M_HA_HN_H}(man)|} = \frac{a^{2\rho}a^{2\rho_H}}{a^{2\rho_H}} = a^{2\rho}.
\end{equation*}
Hence $E\otimes\CC_\chi$ is the representation of $P_H$ given by the character $man\mapsto a^{(\nu+\rho)-(\nu'+\rho_H)}$. We identify the quotient space $V=(T_{(\1,\1)}X)/(T_{(\1,\1)}\calO)$ with $\overline{\frakn}$ by $(X+Y,-X)\mapsto X+Y$. The action of $S_1\cong P_H$ on $\overline{\frakn}$ is given by $p_H\cdot X=\pr_{\overline{\frakn}}\Ad(p_H)X$. It is easy to see that $N_H$ acts trivially and $M_HA_H$ act by the adjoint representation.

\begin{enumerate}
\item For $\FF=\RR,\CC,\HH$ the adjoint action of $M_H=\upS(\Delta\upU(1;\FF)\times\upU(m-1;\FF)\times\upU(n-m;\FF))$ on
$$ V\cong\overline{\frakn}\cong\FF^{n-1}\oplus\Im\FF=\FF^{m-1}\oplus\FF^{n-m}\oplus\Im\FF $$
is given by the defining representations of $\upU(m-1;\FF)$ and $\upU(n-m;\FF)$ on $\FF^{m-1}$ and $\FF^{n-m}$, respectively, and the action of $\Delta\upU(1,\FF)\cong\upU(1;\FF)$ on each factor (see Section~\ref{sec:ExRankOne} for details). Writing
$$ S^r(\overline{\frakn}) = \bigoplus_{j+k+\ell=r}S^j(\FF^{m-1})\otimes S^k(\FF^{n-m})\otimes S^\ell(\Im\FF) $$
we see that there can only be a non-trivial $M_H$-invariant in the summands for $j$ and $k$ even. On each summand $e^{tH_0}$ acts by $e^{-(j+k+2\ell)t}$ and since $j+k+2\ell$ is even we obtain by Theorem~\ref{thm:KV} that
\begin{equation*}
 (\nu+\rho)-(\nu'+\rho_H)\notin(-2\NN_0) \qquad \Rightarrow \qquad \calD'_{\calO_1}(G\times H,E)^L=0.
\end{equation*}
\item For $\FF=\OO$ we have $M_H=\Spin(7)$ acting on $\overline{\frakn}\cong\OO\oplus\Im\OO$ by the direct sum of the spin representation on $\OO\cong\RR^8$ and the natural representation of $\Spin(7)$ on $\Im\OO\cong\RR^7$.  Write
$$ S^r(\overline{\frakn}) = \bigoplus_{j+k=r}S^j(\RR^8)\otimes S^k(\RR^7). $$
As in Section~\ref{sec:Rk1UniquenessO2} we conclude that $j$ must be even if $S^j(\RR^8)$ should contain a non-trivial $\Spin(7)$-invariant. On such a summand $S^j(\RR^8)\otimes S^k(\RR^7)$ the adjoint action of $e^{tH_0}$ is given by $e^{-(j+2k)t}$ with $j+2k$ even. Hence again
\begin{equation*}
 (\nu+\rho)-(\nu'+\rho_H)\notin(-2\NN_0) \qquad \Rightarrow \qquad \calD'_{\calO_1}(G\times H,E)^L=0.
\end{equation*}
\end{enumerate}

\subsubsection{Uniqueness statement}

Altogether we can now prove the following result:

\begin{theorem}\label{thm:UniquenessRank1}
Let $(G,H)=(\SU(1,n;\FF),\upS(\upU(1,m;\FF)\times\upU(n-m;\FF)))$ with $\FF=\RR,\CC,\HH$ and $0<m<n$ or $\FF=\OO$ and $n=2$, $m=1$. For $\nu+\rho-\rho_H\pm\nu'\notin(-2\NN_0)$ the space of continuous $H$-intertwining operators $I^G(\nu)\to I^H(-\nu')'$ is at most one-dimensional. In particular, we have generically
$$ \Hom_H(\pi_\nu|_H,\tau_{\nu'}) = \CC\cdot A(\alpha,\beta) $$
with $\alpha=-(\nu'+\rho_H)$ and $\beta=\frac{(\nu-\rho)+(\nu'+\rho_H)}{2}$.
\end{theorem}

\begin{proof}
Let $A$ and $A'$ be two $H$-intertwining operators $I^G(\nu)|_H\to I^H(-\nu')'$. To these operators we associate $L$-invariant distributions $K_A,K_{A'}\in\calD'(G\times H,E)^L$ as in Lemma~\ref{lem:MapFormsToDistributions}. On the open orbit $\calO_3$ these distributions are by Fact~\ref{fct:InvDistributionsOnHomSpaces} scalar multiples of each other, say $K_A|_{\calO_3}=\lambda K_{A'}|_{\calO_3}$. Then the difference $u:=K_A-\lambda K_{A'}$ vanishes on $\calO_3$. Therefore its restriction to $\calO_2\cup\calO_3$ is contained in $\calD'_{\calO_2}(\calO_2\cup\calO_3,E)^L$ which is trivial by our previous calculations. This means that $u$ vanishes on $\calO_2\cup\calO_3$, so it has support on $\calO_1$. But also $\calD'_{\calO_1}(G\times H,E)^L$ is trivial by the previous considerations and hence $u=0$ which implies $K_A=\lambda K_{A'}$. Therefore $\Hom_H(\pi_\nu|_H,\tau_{\nu'})$ is at most one-dimensional. That it is generically spanned by $A(\alpha,\beta)$ follows from the fact that $A(\alpha,\beta)\in\Hom_H(\pi_\nu|_H,\tau_{\nu'})$ is non-trivial by Theorem~\ref{thm:ConvergenceMeromorphicContinuationIntertwiners}.
\end{proof}

\begin{remark}
Sun--Zhu~\cite{SZ12} recently showed that $(G,H)=(\upO(p,q),\upO(p,q-1))$ and $(G,H)=(\upU(p,q),\upU(p,q-1))$ are strong Gelfand pairs, i.e.\ $\Hom_H(\pi|_H,\tau)$ is at most one-dimensional for all irreducible admissible smooth Fr\'{e}chet representations $\pi$ and $\tau$ which are of moderate growth and $Z(\frakg)$-finite. This implies Theorem~\ref{thm:UniquenessRank1} for the cases $\FF=\RR,\CC$ and $m=n-1$ when $I^G(\nu)$ and $I^H(\nu')$ are irreducible, i.e.\ $\nu\notin\pm(\rho+\NN_0)$ and $\nu'\notin\pm(\rho_H+\NN_0)$. The results of \cite{SZ12} together with Theorem~\ref{thm:UniquenessRank1} imply that the space of intertwining operators $I^G(\nu)\to I^H(\nu')$ can only be of dimension $>1$ if both $\nu\in\pm(\rho+\NN_0)$ and $\nu'\in\pm(\rho_H+\NN_0)$. For $(G,H)=(\upO(1,n),\upO(1,n-1))$ Kobayashi--Speh~\cite{KS,KS14} found all intertwining operators $I^G(\nu)\to I^H(\nu')$, also for singular parameters. For some parameters the space of intertwining operators is two-dimensional.
\end{remark}
\section{Uniqueness for $(G,H)=(\GL(4n,\RR),\GL(2n,\CC))$}

Let $G=\GL(4n,\RR)$ with parabolic subgroup $P=(\GL(2n,\RR)\times\GL(2n,\RR))\ltimes M(2n,\RR)$ and $H=\GL(2n,\CC)\subseteq G$ with parabolic subgroup $P_H=(\GL(n,\CC)\times\GL(n,\CC))\ltimes M(n,\CC)$. Denote by $I^G(\nu)$ and $I^H(\nu')$ the principal series representations of $G$ and $H$ induced from characters $\nu$ of $P$ and $\nu'$ of $P_H$ as in Section~\ref{sec:InducedRepresentations}. In this section we show, using the techniques described in Section~\ref{sec:Uniqueness}, that the space of continuous intertwining operators $I^G(\nu)\to I^H(\nu')$ is generically of dimension $\leq1$.

Realize $H$ as the fixed points in $G$ of the involution
$$ \sigma(g) = JgJ^{-1}, $$
where
$$ J = \left(\begin{array}{cc}J_n&\\&J_n\end{array}\right) \quad \text{and} \quad J_n = \left(\begin{array}{cc}\0_n&\1_n\\-\1_n&\0_n\end{array}\right). $$
Let
$$ P = \left\{\left(\begin{array}{cc}A&B\\ \0&D\end{array}\right):A,D\in\GL(2n,\RR),B\in M(2n,\RR)\right\}. $$
Then $P$ is a $\sigma$-stable parabolic subgroup of $G$ and $P_H=(\GL(n,\CC)\times\GL(n,\CC))\ltimes M(n,\CC)$. We have $M_HA_H\simeq\GL(n,\CC)\times\GL(n,\CC)$ and identifying $\frakn\simeq M(2n,\RR)$ we have
\begin{align}
\begin{split}
 \frakn^{-\sigma} = \left\{\left(\begin{array}{cc}A&B\\B&-A\end{array}\right):A,B\in M(n,\RR)\right\} &\stackrel{\sim}{\to} M(n,\CC)\\
 \left(\begin{array}{cc}A&B\\B&-A\end{array}\right) &\mapsto A+iB.
\end{split}\label{eq:Ex2IdentificationNMinusSigma}
\end{align}
Via this isomorphism the action of $M_HA_H$ on $\frakn^{-\sigma}$ is given by
\begin{equation}
 (g_1,g_2)\cdot X = g_1Xg_2^{-1}, \qquad g_1,g_2\in\GL(n,\CC),X\in M(n,\CC).\label{eq:Ex2ActionMHAHonNMinusSigma}
\end{equation}
Hence $\frakn^{-\sigma}$ decomposes into $M_HA_H$-orbits as follows:
$$ \frakn^{-\sigma} \simeq M(n,\CC)=\bigsqcup_{k=0}^n M(n,\CC)_k, $$
where $M(n,\CC)_k=M_HA_H\cdot e_k$ denotes the subset of matrices of complex rank $k$, a representative being the diagonal matrix $e_k$ with $k$ times $1$ on the diagonal followed by $n-k$ times $0$. In particular $M(n,\CC)_n$ is the unique open orbit and hence $H$ has an open dense orbit on $G/P\times H/P_H$.

\subsection{$P$-orbits in $G/P$}

The Weyl group of $G$ is given by the symmetric group $W=S_{4n}$ on $4n$ elements. Then $W_P=S_{2n}\times S_{2n}$ and hence
$$ W_P\backslash W/W_P = \{W_Pw_kW_P:0\leq k\leq2n\}, \quad w_k=(1,2n+1)\cdots(k,2n+k). $$
Representatives $\tilde{w}_k$ for $w_k$ are given by
$$ \tilde{w}_k = \left(\begin{array}{cc|cc}\0_k&&\1_k&\\&\1_{2n-k}&&\0_{2n-k}\\\hline\1_k&&\0_k&\\&\0_{2n-k}&&\1_{2n-k}\end{array}\right), $$
and hence, by the Bruhat decomposition
$$ G = \bigsqcup_{k=0}^{2n} P\tilde{w}_kP. $$
Identifying $G/P$ with the Grassmannian $\Gr(2n,\RR^{4n})$ of $2n$-dimensional subspaces of $\RR^{4n}$ the $P$-orbits $P\tilde{w}_kP$ are given by
\begin{align*}
 & P\tilde{w}_kP = \left\{L\subseteq\RR^{4n}:\dim L =2n, \dim\left[L\cap\left(\begin{array}{c}\RR^{2n}\\ \0\end{array}\right)\right]=2n-k\right\}\\
 ={}& \left\{\underbrace{\left\{\left(\begin{array}{c}x+by\\y\end{array}\right):x\in L',y\in L''\right\}}_{[L',L'',b]:=}:L'\in\Gr(2n-k,\RR^{2n}),L''\in\Gr(k,\RR^{2n}),b\in M(2n,\RR)\right\}.
\end{align*}
Note that
\begin{equation}
 [L',L'',b]=[\tilde{L}',\tilde{L}'',\tilde{b}] \quad \Leftrightarrow \quad \text{$L'=\tilde{L}'$, $L''=\tilde{L}''$ and $(b-\tilde{b})L''\subseteq L'$.}\label{eq:Ex2EquivalenceRelationTriples}
\end{equation}
The group $P=(\GL(2n,\RR)\times\GL(2n,\RR))\ltimes M(2n,\RR)$ acts on $[L',L'',b]$ by
\begin{align*}
 (g_1,g_2)\cdot[L',L'',b] &= [g_1L',g_2L'',g_1bg_2^{-1}], && g_1,g_2\in\GL(2n,\RR),\\
 X\cdot[L',L'',b] &= [L',L'',b+X], && X\in M(2n,\RR).
\end{align*}

\subsection{$P_H$-orbits in $G/P$}\label{sec:NewExampleOrbit}

To find the $P_H$-orbit decomposition of each $P$-orbit $P\tilde{w}_kP$ we first study the action of the Levi part of $P_H$ on $\Gr(2n-k,\RR^{2n})$ and $\Gr(k,\RR^{2n})$. The $\GL(n,\CC)$-orbit decomposition of $\Gr(k,\RR^{2n})$ is given by
$$ \Gr(k,\RR^{2n}) = \bigsqcup_{0,k-n\leq\ell\leq k/2} \calO_{k,\ell}, $$
where $\calO_{k,\ell}=\{L\in\Gr(k,\RR^{2n}):\dim_\CC (L\cap J_nL)=\ell\}$. A representative in $\calO_{k,\ell}$ is given by the subspace $L_{k,\ell}$ of $\RR^{2n}$ spanned by
$e_1, \dots, e_{k-\ell}, e_{n+1}, \dots, e_{n+\ell}$. The stabilizer of $L_{k,\ell}$ in $\GL(n,\CC)$ is 
\begin{multline*}
 S_{k,\ell} \simeq (\GL(\ell,\CC)\times\GL(k-2\ell,\RR)\times\GL(n-k+\ell,\CC))\\
 \qquad\qquad\ltimes(M(\ell\times(k-2\ell),\CC)\times M((k-2\ell)\times(n-k+\ell),\CC)\times M(\ell\times(n-k+\ell),\CC)).
\end{multline*}
This is a subgroup of a parabolic subgroup of $\GL(n,\CC)$ with Levi component $\GL(\ell,\CC)\times\GL(k-2\ell,\CC)\times\GL(n-k+\ell,\CC)$. For any $[L',L'',b]\in P\tilde{w}_kP$ there exists $(g_1,g_2)\in\GL(n,\CC)\times\GL(n,\CC)=M_HA_H$ such that
$$ (g_1,g_2)\cdot[L',L'',b] = [L_{2n-k,\ell_1},L_{k,\ell_2},b'] $$
for some integers $\ell_1$, $\ell_2$, and $b'\in M(2n,\RR)$ satisfying
\begin{align*}
0,n-k\leq \ell_1\leq n-\frac{k}{2}, \qquad 
0,k-n\leq \ell_2\leq \frac{k}{2}.
\end{align*}
Next, there exists $X\in M(n,\CC)=N_H$ such that
$$ X\cdot[L_{2n-k,\ell_1},L_{k,\ell_2},b'] = [L_{2n-k,\ell_1},L_{k,\ell_2},b''], \qquad \text{where } b'' = \left(\begin{array}{cc}A&B\\B&-A\end{array}\right). $$
Identify $M(2n,\RR)$ with $\frakn$ and $M(n,\CC)$ with $\frakn^{-\sigma}$ as in \eqref{eq:Ex2IdentificationNMinusSigma}.
Then $b''\in\frakn^{-\sigma}$ corresponds to $A+iB\in M(n,\CC)$ on which the action of $M_HA_H=\GL(n,\CC)\times\GL(n,\CC)$ is given by left and right multiplication (see \eqref{eq:Ex2ActionMHAHonNMinusSigma}). Let $\pi$ denote the natural projection $\frakn\to \frakn^{-\sigma}$ along $\frakn_H$.  We further divide $A+iB$ into $3\times 3$ blocks:
\begin{align}\label{eq:ExpressionA+iB}
A+iB=\left(\begin{array}{ccc}C_{11}&C_{12}&C_{13}\\C_{21}&C_{22}&C_{23}\\C_{31}&C_{32}&C_{33}\end{array}\right),
\end{align}
where $C_{11}$, $C_{22}$, and $C_{33}$ are of size $\ell_1\times\ell_2$, $(2n-k-2\ell_1) \times (k-2\ell_2)$ and $(k-n+\ell_1)\times (n-k+\ell_2)$, respectively.
Put $V=\{b\in M(2n,\RR):bL_{k,\ell_2}\subset L_{2n-k,\ell_1}\}$.
Then it follows that
\begin{align*}
\pi(V)=\left\{\left(\begin{array}{ccc}C_{11}&C_{12}&C_{13}\\C_{21}&C_{22}&C_{23}\\
\0&C_{32}&C_{33}\end{array}\right):C_{ij}\text{ arbitrary }\forall i,j\right\}.
\end{align*}
By taking into account the equivalence condition \eqref{eq:Ex2EquivalenceRelationTriples} and the action of $N_H$ this implies that we may assume $C_{ij}=0$ for $(i,j)\neq (3,1)$ in the expression \eqref{eq:ExpressionA+iB}.
Supppose that the rank of $C_{31}$ is $m$.
Then by the multiplication by the
 $\GL(k-n+\ell_1,\CC)\times \GL(\ell_2,\CC)$-component
 of $S_{2n-k,\ell_1}\times S_{k,\ell_2}$, we may take $C_{31}$ to be
\begin{align*}
\left(\begin{array}{cc}
 \1_m & \0 \\
 \0 & \0_{k-n+\ell_1-m,\, \ell_2-m}\\
\end{array}\right).
\end{align*}
Write $b_m$ for the matrix $b''$ given in this way.
It is easy to see that $[L_{2n-k,\ell_1},L_{k,\ell_2},b_m]$ for different $m$
 do not lie in the same $P_H$-orbit.
As a consequence, the $P_H$-orbits in $G/P$ are parameterized by
 four non-negative integers $k,\ell_1,\ell_2,m$ satisfying
\begin{align*}
&k\leq 2n, && 
n-k\leq \ell_1\leq n-\frac{k}{2}, \\ 
&k-n\leq \ell_2\leq \frac{k}{2}, &&
m\leq k-n+\ell_1, \ell_2.
\end{align*}
The representative $[L_{2n-k,\ell_1},L_{k,\ell_2},b_m]\in G/P$
 is the $2n$-dimensional subspace of $\RR^{4n}$ spanned by
\begin{align*}
&e_1,\dots,e_{2n-k-\ell_1}, e_{n+1},\dots,e_{n+\ell_1},
e_{2n-k-\ell_1+1}+e_{2n+1},\dots,e_{2n-k-\ell_1+m}+e_{2n+m},\\
& e_{2n+m+1},\dots,e_{2n+k-\ell_2},
 e_{3n-k-\ell_1+1}-e_{3n+1},\dots,e_{3n-k-\ell_1+m}-e_{3n+m},\\
& e_{3n+m+1},\dots,e_{3n+\ell_2}.
\end{align*}
The unique open orbit corresponds to $(k,\ell_1,\ell_2,m)=(2n,0,n,n)$.

\subsection{Generic uniqueness}

For every $\nu\in\fraka_\CC^*$ the character $P\ni p\mapsto a(p)^\nu$ is given by
\begin{align*}
 \left(\begin{array}{cc}A&X \\ \0&B\end{array}\right)
\mapsto |\det A|^{\nu_1}|\det B|^{\nu_2}, \qquad A,B\in\GL(2n,\RR),X\in M(2n,\RR)
\end{align*}
for some $\nu_1,\nu_2\in \CC$. 
Similarly, for $\nu'\in(\fraka_H)_\CC^*$ the character $P_H\ni p_H\mapsto a(p_H)^{\nu'}$ is given by
\begin{align*}
 \left(\begin{array}{cc}C&Y \\ \0&D\end{array}\right)
\mapsto \textstyle|\det_\CC C|^{\nu'_1}|\det_\CC D|^{\nu'_2}, \qquad C,D\in\GL(n,\CC),Y\in M(n,\CC)
\end{align*}
for some $\nu'_1,\nu'_2\in \CC$.
Here $\det_\CC$ denotes the determinant of complex matrices.
The characters $\rho$ and $\rho_H$ are given by 
\begin{align*}
\left(\begin{array}{cc}A&X \\ \0&B\end{array}\right)
\mapsto |\det A|^{n}|\det B|^{-n},\quad
\left(\begin{array}{cc}C&Y \\ \0&D\end{array}\right)
\mapsto \textstyle|\det_\CC C|^{n}|\det_\CC D|^{-n}.
\end{align*}
For the existence of an $H$-intertwining operator 
$I^G(\nu)\to I^H(-\nu')'$, the subgroup
$$ Z = \{ t\cdot\1_{4n}: t\in\RR^\times \} = Z(G) \subseteq Z(H) $$
has to act by the same scalars on $I^G(\nu)$ and $I^H(\nu')$.
Hence, intertwiners $I^G(\nu)\to I^H(-\nu')'$ can only exist if
$$ 2(\nu_1+\nu_2) = \nu_1'+\nu_2'. $$
Also tensoring the same character of $G$ with $I^G(\nu)$ and with $I^H(\nu')$,
 we may and do assume that $\nu_1+\nu_2=\nu'_1+\nu'_2=0$.
To prove the uniqueness (up to constant multiples) of
 $H$-intertwining operators $I^G(\nu)\to I^H(-\nu')'$
 for generic $\nu_1$ and $\nu'_1$, we apply Theorem~\ref{thm:KV}
 to non-open $H\times P\times P_H$-orbits in $G\times H$.
As in Section~\ref{sec:RankOne}, put $L=H\times P\times P_H$ and $X=G\times H$.
We obtained in Section~\ref{sec:NewExampleOrbit} an $L$-orbit decomposition:
 $X=\bigsqcup_{(k,\ell_1,\ell_2,m)}\calO_{k,\ell_1,\ell_2,m}$.  
Take a non-open orbit $\calO=\calO_{k,\ell_1,\ell_2,m}$ and 
 take an element $g\in G$ such that 
 $gP=[L_{2n-k,\ell_1},L_{k,\ell_2},b_m]\in G/P$.
Let $S$ be the stabilizer of $(g,\1)\in X$ in $L$.
Then 
\begin{align*}
S=\{(p_H, g^{-1} p_H g, p_H)
 : p_H\in P_H\cap g P g^{-1}\}\simeq P_H\cap g P g^{-1}.
\end{align*}
One can see that the elements $p_H\in P_H$ of the following form
 belong to $g P g^{-1}$:
\begin{align*}
p_H=\left(\begin{array}{cc}A&\0 \\ \0&\1_n\end{array}\right),\qquad
 A=\left(\begin{array}{cccc}A_1&&& \\ &A_2&& \\ 
 &&\1_{m} & \\ &&& A_3 \end{array}\right),
\end{align*}
where $A_1\in \GL(\ell_1,\CC)$, $A_2\in \GL(2n-k-2\ell_1,\RR)$, 
 and $A_3\in \GL(k-n+\ell_1-m,\CC)$.
We note that at least one of the integers
 $\ell_1$, $2n-k-2\ell_1$ and $k-n+\ell_1-m$ is positive
 because otherwise $(k,\ell_1,\ell_2,m)=(2n,0,n,n)$, implying $\calO$ is open.
Let $S'\simeq \GL(\ell_1,\CC)\times \GL(2n-k-2\ell_1,\RR)\times \GL(k-n+\ell_1-m,\CC)$
 be the subgroup of $P_H$ consisting of matrices $p_H$ as above.
Then for $p_H\in S'$ it follows that 
\begin{align*}
&a(g^{-1}p_Hg)^{\nu-\rho}
 =\textstyle|\det_\CC A_1|^{2(\nu_1-n)}|\det_\CC A_3|^{-2(\nu_1-n)},\\
&a(p_H)^{-\nu'-\rho_H}
 =\textstyle|\det_\CC A_1\det_\CC A_2\det_\CC A_3|^{-(\nu'_1+n)}.
\end{align*}
Hence the character $\beta$ is given by 
\begin{align*}
\beta(p_H,g^{-1}p_Hg,p_H)=\textstyle|\det_\CC A_1|^{2\nu_1-\nu'_1-3n}|\det_\CC A_2|^{-\nu'_1-n}
|\det_\CC A_3|^{-2\nu_1-\nu'_1+n}
\end{align*}
for $p_H\in S'$.
Let $V:=T_{(g,\1)}X/T_{(g,\1)}\calO$ and
 $\chi(g):=\frac{|\det \Ad_L(g)|}{|\det \Ad_S(g)|}$
 as in Theorem~\ref{thm:KV}.
Here we do not calculate the explicit form of $S$, $V$ or $\chi$. 
However, we can easily deduce an integral condition on the parameters by Theorem~\ref{thm:KV} as follows.
Since $V$ is a real rational representation of $S'$, 
 any one-dimensional $S'$-subrepresentation of $S^r(V_\CC)$
 is of the form:
\begin{align*}
p_H\mapsto
 \textstyle|\det_\CC A_1|^{d_1}|\det_\CC A_2|^{d_2}|\det_\CC A_3|^{d_3}
\end{align*}
for some $d_1, d_2, d_3 \in \ZZ$.
In order to get a sharper condition
 we put $S'':=\{p_H\in S': A_2=\1_{2n-k-2\ell_1}\}$
 and show that $V$ is a complex rational representation
 of $S''$ up to trivial factors, namely,
 there is a decomposition $V=V_1\oplus V_2$ into $S''$-representations
 on real vector spaces such that 
\begin{itemize}
 \item
 $V_1$ has a complex structure for which the representation
 of $S''$ on $V_1$ is complex rational,
 \item
 $S''$ acts trivially on $V_2$. 
\end{itemize}
Indeed, we can endow $T_{(g,\1)}X\simeq \frakg\oplus\frakh$ with
 a complex structure in such a way that $T_{(g,\1)}X$ becomes
 a complex rational representation of $S''$ (or $G'$).
Since any non-trivial irreducible complex rational representation 
 remains irreducible as a real representation
 and $V$ is a quotient $S''$-representation of $T_{(g,\1)}X$,
 there exist a decomposition $V=V_1\oplus V_2$
 and a complex structure on $V_1$ as required.
As a result, the integers $d_1$ and $d_3$ have to be even.
Similarly, since $\chi$ is the top exterior product of $T_{(g,\1)}\calO$, 
 it takes the form:
\begin{align*}
p_H\mapsto 
\textstyle|\det_\CC A_1|^{e_1}|\det_\CC A_2|^{e_2}|\det_\CC A_3|^{e_3}
\end{align*}
for some $e_1, e_3 \in 2\ZZ$ and $e_2\in \ZZ$.
As a result, it is only possible that $\Hom_{S'}(E \otimes \CC_\chi, S^r(V_\CC))\neq 0$ if
 $2\nu_1-\nu'_1 \in (n+2\ZZ)$, $-\nu'_1\in \ZZ$,
 or $-2\nu_1-\nu'_1 \in (n+2\ZZ)$.
By Theorem~\ref{thm:KV} we can conclude generic uniqueness of
 the intertwining operators:

\begin{theorem}\label{thm:UniquenessNewExample}
Let $(G,H)=(\GL(4n;\RR),\GL(2n,\CC))$.
Let $P=(\GL(2n,\RR)\times\GL(2n,\RR))\ltimes M(2n,\RR)$
 and $P_H=(\GL(n,\CC)\times\GL(n,\CC))\ltimes M(n,\CC)$.
Suppose that $P\ni man\mapsto a^\nu$ and
 $P_H\ni m_Ha_Hn_H\mapsto (a_H)^{\nu'}$ are given by
\begin{align*}
 \left(\begin{array}{cc}A&X \\ \0&B\end{array}\right)
\mapsto |\det A|^{\nu_1}|\det B|^{-\nu_1},\quad
 \left(\begin{array}{cc}C&Y \\ \0&D\end{array}\right)
\mapsto \textstyle|\det_\CC C|^{\nu'_1}|\det_\CC D|^{-\nu'_1}
\end{align*}
respectively for some $\nu_1,\nu'_1\in\CC$. 
If $2\nu_1+\nu'_1, 2\nu_1-\nu'_1\notin(n+2\ZZ)$ and $\nu'_1\notin \ZZ$, then
 the space of continuous $H$-intertwining operators $I^G(\nu)\to I^H(-\nu')'$ is at most one-dimensional. In particular, we have generically
$$ \Hom_H(\pi_\nu|_H,\tau_{\nu'}) = \CC\cdot A(\alpha,\beta) $$
with $\alpha=-(\nu'+\rho_H)$ and $\beta=\frac{(\nu-\rho)+(\nu'+\rho_H)}{2}$.
\end{theorem}

\bibliographystyle{amsplain}
\bibliography{bibdb}

\end{document}